\documentclass[final,leqno,onefignum,onetabnum]{siamltexmm}
\usepackage{amsmath}
\usepackage{amsbsy}
\usepackage{amssymb}
\usepackage[dvips]{graphicx}
\usepackage{arydshln}

\title{The $f$-Sensitivity Index
\thanks{This work was supported by the U.S. National Science Foundation under Grant Numbers 1130147 and 1462385.}}

\author{TeX Production\thanks{Society for Industrial and
Applied Mathematics, Philadelphia, Pennsylvania.
(\email{multimedia@siam.org}). Questions, comments, or corrections
to this document may be directed to that email address.}}

\author{Sharif Rahman\thanks{Applied Mathematics \& Computational Sciences, The University of Iowa, Iowa City, IA 52242 (\email{sharif-rahman@uiowa.edu}). Questions, comments, or corrections
to this document may be directed to that email address.}}

\begin{document}

\maketitle
\newcommand{\slugmaster}{%
\slugger{juq}{xxxx}{xx}{x}{x--x}}

\begin{abstract}
This article presents a general multivariate $f$-sensitivity index, rooted in the $f$-divergence between the unconditional and conditional probability measures of a stochastic response, for global sensitivity analysis. Unlike the variance-based Sobol index, the $f$-sensitivity index is applicable to random input following dependent as well as independent probability distributions.  Since the class of $f$-divergences supports a wide variety of divergence or distance measures, a plethora of $f$-sensitivity indices are possible, affording diverse choices to sensitivity analysis. Commonly used sensitivity indices or measures, such as mutual information, squared-loss mutual information, and Borgonovo's importance measure, are shown to be special cases of the proposed sensitivity index.  New theoretical results, revealing fundamental properties of the $f$-sensitivity index and establishing important inequalities, are presented. Three new approximate methods, depending on how the probability densities of a stochastic response are determined, are proposed to estimate the sensitivity index.  Four numerical examples, including a computationally intensive stochastic boundary-value problem, illustrate these methods and explain when one method is more relevant than the others.
\end{abstract}


\pagestyle{myheadings}
\thispagestyle{plain}
\markboth{S. RAHMAN}{THE $f$-SENSITIVITY INDEX}

\section{Introduction}
Complex system modeling and simulation often mandate global sensitivity analysis, which constitutes the study of how the global variation of input, due to its uncertainty, influences the overall uncertain behavior of a response of interest. Most common approaches to sensitivity analysis are firmly anchored in the second-moment properties --- the output variance --- which is divvied up, qualitatively or quantitatively, to distinct sources of input variation \cite{sobol01}. There exist a multitude of methods or techniques for calculating the resultant sensitivity indices of a function of independent variables: the random balance design method \cite{tarantola06}, the state-dependent parameter metamodel \cite{ratto07}, Sobol's method \cite{sobol93}, and the polynomial dimensional decomposition (PDD) method \cite{rahman11}, to name but four. A few methods, such as those presented by Kucherenko, Tarantola, and Annoni \cite{kucherenko12} and Rahman \cite{rahman14}, are also capable of sensitivity analysis entailing correlated or dependent input.

Implicit in the variance-driven global sensitivity analysis is the assumption that the statistical moments satisfactorily describe the stochastic response. In many applications, however, the variance provides a restricted summary of output uncertainty. Therefore, sensitivity indicators stemming solely from the variance should be carefully interpreted. A more rational sensitivity analysis should account for the entire probability distribution of an output variable, meaning that alternative and more appropriate sensitivity indices, based on probabilistic characteristics above and beyond the variance, should be considered. Addressing some of these concerns has led to a sensitivity index by exploiting the ${\cal L}_{1}$ distance between two output probability density functions \cite{borgonovo07}. Such sensitivity analysis establishes a step in the right direction and is founded on the well-known total variational distance between two probability measures. There remain two outstanding research issues for further improvements of density-based sensitivity analysis. First, there is no universal agreement in selecting the total variational distance as the undisputed measure of dissimilarity or affinity between two output probability density functions. In fact, a cornucopia of divergence or distance measures exist in the literature of information theory.  Therefore, a more general framework, in the spirit of density-based measures, should provide diverse choices to sensitivity analysis \cite{borgonovo14,daviega14}.  Second, the density-based sensitivity indices in general are more difficult to calculate than the variance-based sensitivity indices.  This is primarily because the probability density function is harder to estimate than the variance.  Moreover, nearly all estimation methods available today are very expensive due to the existence of the inner and outer integration loops.  Therefore, efficient computational methods for computing density-based sensitivity indices are desirable.

The purpose of this paper is twofold. First, a brief exposition of the $f$-divergence measure is given in section 2, setting the stage for a general multivariate sensitivity index, referred to as the $f$-sensitivity index, presented in section 3.  The section includes new theoretical results representing fundamental properties and important inequalities pertaining to the $f$-sensitivity index.  Second, section 4 introduces three distinct approximate  methods for estimating the $f$-sensitivity index.  The methods depend on how the probability densities of a stochastic response are estimated, including an efficient surrogate approximation commonly used for high-dimensional uncertainty quantification.  Numerical results from three mathematical functions, as well as from a computationally intensive stochastic mechanics problem, are reported in section 5.  Finally, conclusions are drawn in section 6.

\section{$f$-Divergence measure}
Let $\mathbb{N}$, $\mathbb{N}_{0}$, $\mathbb{R}$, and $\mathbb{R}_{0}^{+}$
represent the sets of positive integer (natural), non-negative integer,
real, and non-negative real numbers, respectively. For $k\in\mathbb{N}$,
denote by $\mathbb{R}^{k}$ the $k$-dimensional Euclidean space and by $\mathbb{N}_{0}^{k}$
the $k$-dimensional multi-index space.  These standard notations will be used throughout the paper.

\subsection{Definition}
Let $(\Psi, {\cal G})$ be a measurable space, where $\Psi$ is a sample space and $\cal G$ is a $\sigma$-algebra of the subsets of $\Psi$, satisfying $|{\Psi}|>1$ and $|{\cal G}|>2$, and $\mu$ be a $\sigma$-finite measure on $(\Psi, {\cal G})$.  Let $\cal P$ be a set of all probability measures on $(\Psi, \cal G)$, which are absolutely continuous with respect to $\mu$.  For two such probability measures $P_1,P_2 \in {\cal P}$, let $dP_1/d\mu$ and $dP_2/d\mu$ denote the Radon-Nikodym derivatives of $P_1$ and $P_2$ with respect to the dominating measure $\mu$, that is, $P_1 << \mu$ and $P_2 << \mu$.

Let $f:[0,\infty)\to (-\infty,\infty]$ be an extended real-valued function, which is

\begin{enumerate}
\item
continuous on $[0,\infty)$ and finite-valued on $(0,\infty)$;
\item
convex on $[0,\infty)$, that is, $f(\lambda t_1 + (1-\lambda) t_2)\le \lambda f(t_1) (1-\lambda)f(t_2)$ for any $t_1,t_2 \in [0,\infty)$ and $\lambda \in [0,1]$;
\footnote{Trivial convex functions of the form $c_1+c_2t$, $c_1,c_2 \in \mathbb{R}$, are excluded.}
\item
strictly convex at $t=1$, that is, $f(1) < \lambda f(t_1) + (1-\lambda)f(t_2)$ for any $t_1,t_2 \in [0,\infty)\setminus \{1\}$ and $\lambda \in (0,1)$ such that $\lambda t_1 + (1-\lambda) t_2=1$;
\footnote{Geometrically, strict convexity means that each chord linking two function values $f(t_1)$ and $f(t_2)$ evaluated on two sides of the point $t=1$ on the graph of $f$ lies above the function value $f(1)$.}; and
\item
equal to \emph{zero} at $t=1$, that is, $f(1)=0$.
\end{enumerate}
The $f$-divergence, describing the difference or discrimination between two probability measures $P_1$ and $P_2$, is defined by the integral
\begin{equation}
D_f\left( P_1 \parallel P_2\right):=
\int_{\Psi} f\left(\dfrac{dP_1/d\mu}{dP_2/d\mu}\right) \frac{dP_2}{d\mu} d\mu,
\label{1}
\end{equation}
provided that the undefined expressions are interpreted by \cite{csiszar67,csiszar63}
\begin{equation*}
f(0)=\lim_{t \to 0_+} f(t),~~0 \cdot f\left( \frac{0}{0} \right) = 0,
\end{equation*}
\begin{equation*}
0 \cdot f\left( \frac{a}{0} \right)=\lim_{\epsilon \to 0_+} \epsilon f \left( \frac{a}{\epsilon} \right)=
a \lim_{t \to \infty} \frac{f(t)}{t}~~(0 < a < \infty).
\end{equation*}
To define the $f$-divergence for absolutely continuous probability measures in terms of elementary probability theory, take $\Psi$ to be the real line and $\mu$ to be the Lebesgue measure, that is, $d\mu=d\xi$,  $\xi \in \mathbb R$, so that $dP_1/d\mu$ and $dP_2/d\mu$ are simply probability density functions, denoted by $f_1(\xi)$ and $f_2(\xi)$, respectively. Then the $f$-divergence can also be defined by
\begin{equation}
D_f\left( P_1 \parallel P_2\right):=
\int_{\mathbb{R}}f\left(\dfrac{f_1(\xi)}{f_2(\xi)}\right) f_2(\xi)d\xi.
\label{2}
\end{equation}
The divergence measures in (\ref{1}) and (\ref{2}) were introduced in the 1960s by Csisz\'ar \cite{csiszar67,csiszar63}, Ali and Silvey \cite{ali66}, and Morimoto \cite{morimoto63}.  Similar definitions exist for discrete probability measures.

\subsection{General properties}
Vajda \cite{vajda72}, Liese and Vajda \cite{liese87}, and \"{O}sterreicher \cite{osterreicher02} discussed general properties of the $f$-divergence measure, including a few axiomatic ones. The basic but important properties are as follows \cite{liese87,osterreicher02,vajda72}: \\

\begin{enumerate}
\item
\emph{Non-negativity and reflexivity:} $D_f(P_1 \parallel P_2)\ge 0$ with equality if and only if $P_1=P_2$. \\

\item
\emph{Duality:} $D_f(P_1 \parallel P_2)=D_{f^*}(P_2 \parallel P_1)$, where $f^*(t):=tf(1/t)$, $t \in (0,\infty)$, is the *-conjugate (convex) function of $f$. When $f(t)=f^*(t)$, $f$ is *-self conjugate. \\

\item
\emph{Invariance:} $D_f(P_1 \parallel P_2)=D_{f^\dagger}(P_1 \parallel P_2)$, where $f^\dagger(t)=f(t)+c(t-1)$, $c \in \mathbb{R}$. \\

\item
\emph{Symmetry:} $D_f(P_1 \parallel P_2)=D_f(P_2 \parallel P_1)$ if and only if $f^*(t)-f(t)=c(t-1)$, where $f^*(t):=tf(1/t)$, $t \in (0,\infty)$, and $c \in \mathbb{R}$. When $c=0$, the symmetry and duality properties coincide. \\

\item
\emph{Range of Values:} $0 \le D_f(P_1 \parallel P_2)\le f(0)+f^*(0)$, where  $f(0)=\lim_{t \to 0_+} f(t)$ and $f^*(0)=\lim_{t \to 0_+} tf(1/t)=\lim_{t \to \infty} f(t)/t$. The left equality holds if and only if $P_1=P_2$. The right equality holds if and only if $P_1 \perp P_2$, that is, for mutually singular (orthogonal) measures, and is attained when $f(0)+f^*(0) < \infty$.  \\
\end{enumerate}

The normalization condition $f(1)=0$ is commonly adopted to ensure that the smallest possible value of $D_f(P_1 \parallel P_2)$ is \emph{zero}.  But fulfilling such condition by the class $\mathbb{F}$ of convex functions $f$ is not required.  This is because, for the subclass $\mathbb{F}' \subset \mathbb{F}$ such that $f \in \mathbb{F}'$ satisfies $f(1)=0$, the shift by the constant $-f(1)$ sends every $f \in \mathbb{F} - \mathbb{F}'$ to $\mathbb{F}'$.  Indeed, some of these properties may still hold if $f(1) \ne 0$ or if $f$ is not restricted to the convexity properties.

Depending on how $f$ is defined, the $f$-divergence may or may not be a true metric.  For instance, it is not necessarily symmetric in $P_1$ and $P_2$ for an arbitrary convex function $f$; that is, the $f$-divergence from $P_1$ to $P_2$ is generally not the same as that from $P_2$ to $P_1$, although it can be easily symmetrized when required.  Furthermore, the $f$-divergence does not necessarily satisfy the triangle inequality.

\subsection{Common instances}
It is well known that $D_f$ has a versatile functional form, resulting in a number of popular information divergence measures.  Indeed, many of the well-known divergences or distances commonly used in information theory and statistics are easily reproduced by appropriately selecting the generating function $f$. Familiar examples of the $f$-divergence include the forward and reversed Kullback-Leibler divergences $D_{KL}$ and $D_{KL'}$ \cite{kullback51}, Kolmogorov total variational distance $D_{TV}$ \cite{kolmogorov57}, Hellinger distance $D_H$ \cite{hellinger09}, Pearson $\chi^2$ divergence $D_P$ \cite{daviega14,pearson00}, Neyman $\chi^2$ divergence $D_N$ \cite{daviega14,pearson00}, $\alpha$ divergence $D_\alpha$ \cite{dragomir00}, Vajda $\chi^\alpha$ divergence $D_V$ \cite{vajda72}, Jeffreys distance $D_J$ \cite{jeffreys46}, and triangular discrimination $D_\triangle$ \cite{topsoe99}, to name a few, and are defined as
\begin{subequations}
\begin{equation}
D_{KL}\left( P_1 \parallel P_2\right):=
\int_{\mathbb{R}}f_1(\xi) \ln \left[ \dfrac{f_1(\xi)}{f_2(\xi)} \right] d\xi,
\label{3a}
\end{equation}
\begin{equation}
D_{KL'}\left( P_1 \parallel P_2\right):=
\int_{\mathbb{R}}f_2(\xi) \ln \left[ \dfrac{f_2(\xi)}{f_1(\xi)} \right] d\xi
=: D_{KL}\left( P_2 \parallel P_1\right),
\label{3b}
\end{equation}
\begin{equation}
D_{TV}\left( P_1 \parallel P_2\right):=
\int_{\mathbb{R}}\left| f_1(\xi)-f_2(\xi)\right| d\xi,
\label{3c}
\end{equation}
\begin{equation}
D_{H}\left( P_1 \parallel P_2\right):=
\int_{\mathbb{R}}\left[ \sqrt{f_1(\xi)}-\sqrt{f_2(\xi)} \right]^2 d\xi,
\label{3d}
\end{equation}
\begin{equation}
D_{P}\left( P_1 \parallel P_2\right):=
\int_{\mathbb{R}}f_2(\xi) \left[\dfrac{f_1^2(\xi)}{f_2^2(\xi)} -1 \right] d\xi,
\label{3e}
\end{equation}
\begin{equation}
D_{N}\left( P_1 \parallel P_2\right):=
\int_{\mathbb{R}}f_1(\xi) \left[\dfrac{f_2^2(\xi)}{f_1^2(\xi)} -1 \right] d\xi
:= D_{P}\left( P_2 \parallel P_1\right),
\label{3f}
\end{equation}
\begin{equation}
D_{\alpha}\left( P_1 \parallel P_2\right):=
{\displaystyle \frac{4}{1-\alpha^2}} \left[ 1 -
\int_{\mathbb{R}} {f_1}^{(1-\alpha)/2}(\xi) {f_2}^{(1+\alpha)/2}(\xi) d\xi \right],
~\alpha \in \mathbb{R} \setminus \{\pm 1\},
\label{3g}
\end{equation}
\begin{equation}
D_V \left( P_1 \parallel P_2\right):=
\int_{\mathbb{R}}f_2^{1-\alpha}(\xi)\left| f_1(\xi) - f_2(\xi) \right|^\alpha d\xi,
~\alpha \ge 1,
\label{3h}
\end{equation}
\begin{equation}
D_{J}\left( P_1 \parallel P_2\right):=
\int_{\mathbb{R}}\left[ f_1(\xi)-f_2(\xi) \right] \ln \left[ \dfrac{f_1(\xi)}{f_2(\xi)} \right] d\xi,
\label{3i}
\end{equation}
\begin{equation}
D_{\triangle}\left( P_1 \parallel P_2\right):=
\int_{\mathbb{R}}\dfrac{\left[f_1(\xi)-f_2(\xi)\right]^2}{f_1(\xi)+f_2(\xi)} d\xi.
\label{3j}
\end{equation}
\end{subequations}
The definitions of some of these divergences, notably the two Kullback-Leibler and Pearson-Neyman $\chi^2$ divergences, are inverted when the $f$-divergence is defined by swapping $P_1$ and $P_2$ in (\ref{1}) or (\ref{2}). There are also many other information divergence measures that are not subsumed by the $f$-divergence measure.  See the paper by Kapur \cite{kapur84} or the book by Taneja \cite{taneja01}.  Nonetheless, any of the divergence measures from the class of $f$-divergences or others can be exploited for sensitivity analysis, as described in the following section.

\section{$f$-Sensitivity index}
Let $(\Omega,\mathcal{F},P)$ be a complete probability space, where
$\Omega$ is a sample space, $\mathcal{F}$ is a $\sigma$-field on
$\Omega$, and $P:\mathcal{F}\to[0,1]$ is a probability measure.
With $\mathcal{B}^{N}$ representing the Borel $\sigma$-field on
$\mathbb{R}^{N}$, $N\in\mathbb{N}$, consider an $\mathbb{R}^{N}$-valued absolutely continuous
random vector $\mathbf{X}:=(X_{1},\cdots,X_{N}):(\Omega,\mathcal{F})\to(\mathbb{R}^{N},\mathcal{B}^{N})$,
describing the statistical uncertainties in all system and input parameters
of a general stochastic problem. The probability law of $\mathbf{X}$, which may comprise independent or dependent random variables, is completely defined by its joint probability density function $f_{\mathbf{X}}:\mathbb{R}^{N}\to\mathbb{R}_{0}^{+}$.
Let $u$ be a non-empty subset of $\{1,\cdots,N\}$ with the complementary set
$-u:=\{1,\cdots,N\} \setminus u$ and cardinality $1\le|u|\le N$,
and let $\mathbf{X}_{u}=(X_{i_{1}},\cdots,X_{i_{|u|}})$,
$1\leq i_{1}<\cdots<i_{|u|}\leq N$, be a subvector of $\mathbf{X}$
with $\mathbf{X}_{-u}:=\mathbf{X}_{\{1,\cdots,N\} \setminus u}$ defining
its complementary subvector. Then, for a given $\emptyset\neq u\subseteq\{1,\cdots,N\}$,
the marginal density function of $\mathbf{X}_{u}$ is $f_{\mathbf{X}_u}(\mathbf{x}_{u}):=\int_{\mathbb{R}^{N-|u|}}f_{\mathbf{X}}(\mathbf{x})d\mathbf{x}_{-u}$.

\subsection{Definition}
Let $y(\mathbf{X}):=y(X_{1},\cdots,X_{N}$), a real-valued, continuous, measurable transformation on $(\Omega,\mathcal{F})$, define a general stochastic response of interest. Define $Y:=y(\mathbf{X})$ to be the associated output random variable. For global sensitivity analysis, suppose that the sensitivity of $Y$ with respect to a subset
$\mathbf{X}_u$, $\emptyset\neq u\subseteq\{1,\cdots,N\}$, of input variables $\mathbf{X}$ is desired. As shown by a few researchers for univariate cases only \cite{borgonovo14,daviega14}, such a multivariate sensitivity measure can be linked to the divergence between the unconditional and conditional probability measures of $Y$.  Denote by $P_Y$ and $P_{Y|\mathbf{X}_u}$ the probability measures and by $f_{Y}(\cdot)$ and $f_{Y|\mathbf{X}_{u}}(\cdot|\cdot)$ the probability density functions of random variables $Y$ and $Y|\mathbf{X}_u$, respectively, where $Y|\mathbf{X}_u$ stands for $Y$ conditional on $\mathbf{X}_u$, which is itself random.  Setting $P_1=P_Y$, $f_1=f_Y$, $P_2=P_{Y|\mathbf{X}_u}$, and $f_2=f_{Y|\mathbf{X}_u}$ in (\ref{2}), the $f$-divergence becomes
\begin{equation}
D_f\left( P_Y \parallel P_{Y|\mathbf{X}_u}\right):=
\int_{\mathbb{R}}f\left(\dfrac{f_{Y}(\xi)}
{f_{Y|\mathbf{X}_{u}}(\xi|\mathbf{x}_u)}\right) f_{Y|\mathbf{X}_{u}}(\xi|\mathbf{x}_u)d\xi.
\label{4}
\end{equation}
As explained in the preceding section, $D_f(P_Y \parallel P_{Y|\mathbf{X}_u})$ characterizes the discrimination between $P_Y$ and $P_{Y|\mathbf{X}_u}$, but it is random because $\mathbf{X}_u$ is random.

\begin{definition}
A general multivariate $f$-sensitivity index of an output random variable $Y$ for a subset $\mathbf{X}_{u}$, $\emptyset\neq u\subseteq\{1,\cdots,N\}$, of input random variables $\mathbf{X}:=(X_1,\cdots,X_N)$, denoted by $H_{u,f}$, is defined as the expected value of the $f$-divergence from $P_Y$ to $P_{Y|\mathbf{X}_u}$, that is,
\begin{equation}
H_{u,f}:=\mathbb{E}_{\mathbf{X}_u} \left[ D_f\left( P_Y \parallel P_{Y|\mathbf{X}_u}\right) \right],
\label{4b}
\end{equation}
where $\mathbb{E}_{\mathbf{X}_u}$ is the expectation operator with respect to the probability measure of $\mathbf{X}_{u}$.
\end{definition}

From the definition of the expectation operator, the $f$-sensitivity index
\begin{equation}
\begin{array}{rcl}
H_{u,f} & = &
\int_{\mathbb{R}^{|u|}}
 D_f\left( P_Y \parallel P_{Y|\mathbf{X}_u=\mathbf{x}_u}\right)
f_{\mathbf{X}_u}(\mathbf{x}_u)d{\mathbf{x}_u} \\
        & = &
\int_{\mathbb{R}^{|u|}\times\mathbb{R}}
f\left(
\dfrac{f_{Y}(\xi)}{f_{Y|\mathbf{X}_u}(\xi|\mathbf{x}_u)}
\right)
f_{Y|\mathbf{X}_u}(\xi|\mathbf{x}_u)f_{\mathbf{X}_u}(\mathbf{x}_u)d{\mathbf{x}_u}d\xi \\

        & =  &
\int_{\mathbb{R}^{|u|}\times\mathbb{R}}
f\left(
\dfrac{f_{Y}(\xi)f_{\mathbf{X}_u}(\mathbf{x}_u)}
{f_{\mathbf{X}_u,Y}(\mathbf{x}_u,\xi)}
\right)
f_{\mathbf{X}_{u},Y}(\mathbf{x}_u,\xi)d{\mathbf{x}_u}d\xi,
\end{array}
\label{5}
\end{equation}
where $P_{Y|\mathbf{X}_u=\mathbf{x}_u}$ and $f_{Y|\mathbf{X}_u}(\xi|\mathbf{x}_u)$ are the probability measure and probability density function, respectively, of $Y$ conditional on $\mathbf{X}_u=\mathbf{x}_u$, and $f_{\mathbf{X}_{u},Y}(\mathbf{x}_u,\xi)$ is the joint probability density function of $(\mathbf{X}_u,Y)$.  The last equality in (\ref{5}) is formed by the recognition that $f_{\mathbf{X}_u,Y}(\mathbf{x}_u,\xi)=f_{Y|\mathbf{X}_u}(\xi|\mathbf{x}_u)f_{\mathbf{X}_u}(\mathbf{x}_u)$ and is useful for calculating the sensitivity index, to be discussed in section 4.

For variance-based sensitivity analysis entailing independent random variables, there exists a well-known importance measure, namely, the Sobol index.  One way to explain the Sobol index is the analysis-of-variance (ANOVA) decomposition of a square-integrable function $y$, expressed by the compact form \cite{rahman14,sobol01,sobol93}
\begin{subequations}
\begin{align}
y(\mathbf{X})         & = {\displaystyle \sum_{u\subseteq\{1,\cdots,N\}}y_{u}(\mathbf{X}_{u})},
\label{5a1}                \\
y_{\emptyset}         & = \int_{\mathbb{R}^{N}}y(\mathbf{x})\prod_{i=1}^{N}{\displaystyle {\displaystyle {\textstyle f_{X_i}(x_{i})}}dx_{i}},
\label{5a2}                \\
y_{u}(\mathbf{X}_{u}) & = {\displaystyle \int_{\mathbb{R}^{N-|u|}}y(\mathbf{X}_{u},\mathbf{x}_{-u})}\prod_{i=1,i\notin u}^{N}{\displaystyle {\displaystyle {\textstyle f_{X_i}(x_{i})}}dx_{i}}-{\displaystyle \sum_{v\subset u}}y_{v}(\mathbf{X}_{v}),
\label{5a3}
\end{align}
\label{5a}
\end{subequations}
\!\!which is a finite, hierarchical expansion in terms of its input variables with increasing dimensions.  Here, $y_{u}$ is a $|u|$-variate component function describing a constant or the interactive effect of $\mathbf{X}_{u}$ on $y$ when $|u|=0$ or $|u|>0$. The summation in (\ref{5a1}) comprises $2^{N}$ component functions, with each function depending on a group of variables indexed by a particular subset of $\{1,\cdots,N\}$, including the empty set $\emptyset$.  Applying the expectation operator $\mathbb{E}_\mathbf{X}$ on $y(\mathbf{X})$ and its square from (\ref{5a1}) and recognizing the \emph{zero}-mean and orthogonal properties of $y_{u}(\mathbf{X}_{u})$,
$\emptyset\ne u\subseteq\{1,\cdots,N\}$, in (\ref{5a3}) \cite{rahman14}, the variance
\begin{equation}
\sigma^2:=\mathbb{E}_\mathbf{X}\left[Y-\mathbb{E}_\mathbf{X}[Y]\right]^2=
\sum_{\emptyset\ne u\subseteq\{1,\cdots,N\}} {\sigma_u}^2
\label{5b}
\end{equation}
of $Y$ splits into partial variances
\begin{equation}
{\sigma_u}^2:=\mathbb{E}_\mathbf{X}\left[y_{u}^{2}(\mathbf{X}_{u})\right]=
\mathbb{E}_{\mathbf{X}_{u}}\left[y_{u}^{2}(\mathbf{X}_{u})\right],
~\emptyset\ne u\subseteq\{1,\cdots,N\},
\label{5c}
\end{equation}
of all non-constant ANOVA component functions. Henceforth, the Sobol sensitivity index of $Y$ for a subset of variables $\mathbf{X}_u$ is defined as \cite{sobol01}
\begin{equation}
S_u := \dfrac{\sigma_u^2}{\sigma^2},
\label{5d}
\end{equation}
provided that $0 < \sigma^2 < \infty$.  The Sobol index is bounded between 0 and 1 and represents the fraction of the variance of $Y$ contributed by the $|u|$-variate interaction of input variables $\mathbf{X}_u$.  There exist $2^N-1$ such indices, adding up to $\sum_{\emptyset\ne u\subseteq\{1,\cdots,N\}} S_u = 1$.

Does the $f$-sensitivity index provide a more useful insight than the existing variance-based Sobol index into the importance of input variables?  To answer this question, consider a purely additive function $y(\mathbf{X})=a+\sum_{i=1}^N X_i$, where $a \in \mathbb{R}$ is an arbitrary real-valued constant and the input random variables $X_i$, $i=1,\cdots,N$, have \emph{zero} means and identical variances $\mathbb{E}_{X_i}[X_i^2]=s^2$, $0 < s^2 < \infty$, but otherwise follow independent and arbitrary probability distributions.  Then, from (\ref{5a}) through (\ref{5d}), (1) the ANOVA component functions $y_\emptyset=a$, $y_{\{i\}}(X_i)=X_i$, $i=1,\cdots,N$, and $y_{u}(\mathbf{X}_u)=0$ for $2\le |u| \le N$; (2) the variances $\sigma^2=Ns^2$, $\sigma_{\{i\}}^2=s^2$, $i=1,\cdots,N$, and $\sigma_{u}^2=0$ for $2\le |u| \le N$; and (3) the Sobol indices $S_{\{i\}}=1/N$, $i=1,\cdots,N$, and $S_u=0$ for $2 \le |u| \le N$.  As all univariate Sobol indices are the same, so are the contributions of input variables to the variance of $Y$.  Hence, according to the Sobol index, all input variables are equally important, regardless of their probability distributions. This is unrealistic, but possible because the variance is just a moment and provides only a partial description of the uncertainty of an output variable.  In contrast, the $f$-sensitivity indices will vary depending on the choice of the input density functions, therefore, providing a more rational measure of the influence of input variables.

\subsection{Fundamental properties}
It is important to derive and emphasize the fundamental properties of the $f$-sensitivity index $H_{u,f}$ inherited from the $f$-divergence measure. The properties, including a few important inequalities, are described in conjunction with six propositions as follows.

\begin{proposition}[Non-negativity]
The $f$-sensitivity index $H_{u,f}$ of $Y$ for $\mathbf{X}_u$, $\emptyset\neq u\subseteq\{1,\cdots,N\}$, is non-negative and vanishes when $Y$ and $\mathbf{X}_u$ are statistically independent.
\label{p1}
\end{proposition}

\begin{proof}
Since $D_f(P_Y \parallel P_{Y|\mathbf{X}_u=\mathbf{x}_u}) \ge 0$ by virtue of the non-negativity property of the $f$-divergence and $f_{\mathbf{X}_u}(\mathbf{x}_u)\ge 0$ for any $\mathbf{x}_u \in \mathbb{R}^{|u|}$, the first line of (\ref{5}) yields
\begin{equation*}
H_{u,f} = \int_{\mathbb{R}^{|u|}}
D_f\left( P_Y \parallel P_{Y|\mathbf{X}_u=\mathbf{x}_u}\right)
f_{\mathbf{X}_u}(\mathbf{x}_u)d{\mathbf{x}_u} \ge 0,
\end{equation*}
proving the first part of the proposition.  If $Y$ and $\mathbf{X}_u$ are statistically independent, then $P_Y = P_{Y|\mathbf{X}_u=\mathbf{x}_u}$ for any $\mathbf{x}_u \in \mathbb{R}^N$, resulting in $D_f( P_Y \parallel P_{Y|\mathbf{X}_u=\mathbf{x}_u})=D_f( P_Y \parallel P_Y)=0$, owing to the reflexivity property or the range of values (left equality) of the $f$-divergence. In that case,
\begin{equation*}
H_{u,f} = \int_{\mathbb{R}^{|u|}}
D_f\left( P_Y \parallel P_{Y|\mathbf{X}_u=\mathbf{x}_u}\right)
f_{\mathbf{X}_u}(\mathbf{x}_u)d{\mathbf{x}_u} =
\int_{\mathbb{R}^{|u|}}
D_f\left( P_Y \parallel P_Y\right)
f_{\mathbf{X}_u}(\mathbf{x}_u)d{\mathbf{x}_u} = 0,
\end{equation*}
proving the second part of the proposition.
\end{proof}

\begin{proposition}[Range of values]
The range of values of $H_{u,f}$ is
\begin{equation*}
0 \le H_{u,f} \le f(0)+f^*(0),
\end{equation*}
where $f(0)=\lim_{t \to 0_+} f(t)$ and $f^*(0)=\lim_{t \to 0_+} tf(1/t)=\lim_{t \to \infty} f(t)/t$.
\label{p2}
\end{proposition}

\begin{proof}
See the proof of Proposition \ref{p1} for the left inequality.  The right inequality is derived from the largest value of $D_f( P_Y \parallel P_{Y|\mathbf{X}_u})$, which is $f(0)+f^*(0)$, according to the range of values (right equality) of the $f$-divergence. Therefore, (\ref{4b}) yields
\begin{equation*}
H_{u,f} := \mathbb{E}_{\mathbf{X}_u} \left[ D_f\left( P_Y \parallel P_{Y|\mathbf{X}_u}\right) \right]
\le \mathbb{E}_{\mathbf{X}_u} \left[ f(0)+f^*(0) \right] = f(0)+f^*(0),
\end{equation*}
completing the proof.
\end{proof}

From Proposition \ref{p2}, $H_{u,f}$ has a sharp lower bound, which is \emph{zero} since $f(1)=0$.  In contrast, $H_{u,f}$ may or may not have an upper bound, depending on whether $f(0)+f^*(0)$ is finite or infinite.  If there is an upper bound, then the largest value $f(0)+f^*(0)$ is a sharp upper bound, and hence can be used to scale $H_{u,f}$ to vary between 0 and 1. For instance, when $f=|t-1|$, the result is the well-known variational distance measure $D_{TV}(P_Y \parallel P_{Y|\mathbf{X}_u})$ and the upper bound of the associated sensitivity index $H_{u,TV}$ (say) is $f(0)+f^*(0)=1+1=2$. When $f=t \ln t$ or $f=-\ln t$, then $f(0)+f^*(0)=\infty$, meaning that the sensitivity index  $H_{u,KL}$ (say) or $H_{u,KL'}$ (say), derived from the Kullback-Leibler divergence measure $D_{KL}(P_Y \parallel P_{Y|\mathbf{X}_u})$ or
$D_{KL'}(P_Y \parallel P_{Y|\mathbf{X}_u})$, has no upper bound.  No scaling is possible in such a case.

\begin{proposition}[Importance of all input variables]
The $f$-sensitivity index $H_{\{1,\cdots,N\},f}$ of $Y$ for all input variables $\mathbf{X}=(X_1,\cdots,X_N)$ is
\begin{equation*}
H_{\{1,\cdots,N\},f} =  f(0)+f^*(0),
\end{equation*}
where $f(0)=\lim_{t \to 0_+} f(t)$ and $f^*(0)=\lim_{t \to 0_+} tf(1/t)=\lim_{t \to \infty} f(t)/t$.
\label{p3}
\end{proposition}

\begin{proof}
The probability measure $P_{Y|\mathbf{X}=\mathbf{x}}$ is a dirac measure, representing an almost sure outcome $\xi=y(\mathbf{x})$, where $\mathbf{x} \in \mathbb{R}^N$ and $\xi \in \mathbb{R}$.  Decompose $\mathbb{R}$ into two disjoint subsets  $\mathbb{R} \setminus \{\xi\}$ and $\{\xi\}$ and observe that
\begin{equation*}
P_{Y|\mathbf{X}=\mathbf{x}}\left(\mathbb{R} \setminus \{\xi\}\right) =  P_{Y}\left( \{\xi\} \right) = 0.
\end{equation*}
Therefore, the probability measures $P_Y$ and $P_{Y|\mathbf{X}=\mathbf{x}}$ are mutually singular (orthogonal), that is, $P_Y \perp P_{Y|\mathbf{X}=\mathbf{x}}$.  Consequently, $D_f( P_Y \parallel P_{Y|\mathbf{X}})=f(0)+f^*(0)$, according to the range of values (right equality) of the $f$-divergence. Finally, for $u=\{1,\cdots,N\}$, (\ref{4b}) yields
\begin{equation*}
H_{\{1,\cdots,N\},f} =
\mathbb{E}_{\mathbf{X}} \left[ D_f\left( P_Y \parallel P_{Y|\mathbf{X}}\right) \right] =
\mathbb{E}_{\mathbf{X}} \left[ f(0)+f^*(0) \right] = f(0)+f^*(0).
\end{equation*}
\end{proof}

For the special case of $f=|t-1|$, the index derived from the total variational distance $H_{\{1,\cdots,N\},{TV}}=2$. Therefore, when normalized, $H_{\{1,\cdots,N\},{TV}}/2=1$, which is the same value reported by Borgonovo \cite{borgonovo07}.

\begin{proposition}[Importance for an independent subset]
Let $\emptyset\neq u, v\subset\{1,\cdots,N\}$.  If $Y$ and $\mathbf{X}_v$ are statistically independent, then
\begin{equation*}
H_{u \cup v,f} = H_{u \setminus v,f}.
\end{equation*}
\label{p4}
In addition, if $u$ and $v$ are disjoint subsets, that is, $u \cap v = \emptyset$, then
\begin{equation*}
H_{u \cup v,f} = H_{u,f}.
\end{equation*}
\label{p4b}
\end{proposition}

\begin{proof}
For any $\emptyset\neq u, v\subset\{1,\cdots,N\}$, observe that $u \cup v=(u \setminus v) \cup v$ and $(u \setminus v)\cap v = \emptyset$.  Since $Y$ is independent of $\mathbf{X}_v$, the probability measures $P_{Y|\mathbf{X}_{u \cup v}=\mathbf{x}_{u \cup v}}$ and $P_{Y|\mathbf{X}_{u \setminus v}=\mathbf{x}_{u\setminus v}}$ are the same, yielding $D_f( P_Y \parallel P_{Y|\mathbf{X}_{u \cup v}=\mathbf{x}_{u \cup v}})=D_f( P_Y \parallel P_{Y|\mathbf{X}_{u \setminus v}=\mathbf{x}_{u\setminus v}})$.  Applying this condition to the expression of $H_{u \cup v,f}$ --- the $f$-sensitivity index of $Y$ for $\mathbf{X}_{u \cup v}$ --- in the first line of (\ref{5}) and noting $d\mathbf{x}_{u \cup v}=d{\mathbf{x}_{v}}d{\mathbf{x}_{u \setminus v}}$ results in
\begin{equation*}
\begin{array}{rcl}
H_{u \cup v,f} & = &
\int_{\mathbb{R}^{|u \cup v|}}
D_f\left( P_Y \parallel P_{Y|\mathbf{X}_{u \cup v}=\mathbf{x}_{u \cup v}}\right)
f_{\mathbf{X}_{u \cup v}}(\mathbf{x}_{u \cup v})d{\mathbf{x}_{u \cup v}} \\
               & =  &
\int_{\mathbb{R}^{|u \setminus v|}}
D_f\left( P_Y \parallel P_{Y|\mathbf{X}_{u \setminus v}=\mathbf{x}_{u \setminus v}}\right)
\left(
\int_{\mathbb{R}^{|v|}}
f_{\mathbf{X}_{(u \setminus v) \cup v}}(\mathbf{x}_{(u \setminus v)\cup v})
d{\mathbf{x}_{v}}
\right)
d{\mathbf{x}_{u \setminus v}}  \\
                & = &
\int_{\mathbb{R}^{|u \setminus v|}}
D_f\left( P_Y \parallel P_{Y|\mathbf{X}_{u \setminus v}=\mathbf{x}_{u \setminus v}}\right)
f_{\mathbf{X}_{u \setminus v}}(\mathbf{x}_{u \setminus v})
d{\mathbf{x}_{u \setminus v}}  \\
                & = &
H_{u \setminus v,f},
\end{array}
\end{equation*}
proving the first part of the proposition. Here, the second equality is obtained by recognizing that
$D_f( P_Y \parallel P_{Y|\mathbf{X}_{u \setminus v}=\mathbf{x}_{u \setminus v}})$ does not depend on $\mathbf{x}_v$ and $f_{\mathbf{X}_{u \cup v}}(\mathbf{x}_{u \cup v})=f_{\mathbf{X}_{(u \setminus v) \cup v}}(\mathbf{x}_{(u \setminus v)\cup v})$.  The third equality is attained by integrating out $f_{\mathbf{X}_{(u \setminus v) \cup v}}(\mathbf{x}_{(u \setminus v)\cup v})$  with respect to $\mathbf{x}_v$ on $\mathbb{R}^{|v|}$, resulting in $f_{\mathbf{X}_{u \setminus v}}(\mathbf{x}_{u \setminus v}):=\int_{\mathbb{R}^{|v|}}f_{\mathbf{X}_{(u \setminus v) \cup v}}(\mathbf{x}_{(u \setminus v)\cup v})d{\mathbf{x}_v}$.  The second part of the proposition results from the reduction, $H_{u \setminus v,f}=H_{u,f}$, when $u \cap v = \emptyset$.
\end{proof}

As a special case, consider $u=\{i\}$ and $v=\{j\}$, where $i,j = 1,\cdots,N$, $i \ne j$.  Then, according to Proposition \ref{p4}, $H_{\{i,j\},f}=H_{\{i\},f}$, meaning that there is no contribution of $X_j$ to the sensitivity of $Y$ for $(X_i,X_j)$ if $y$ does not depend on $X_j$.

\begin{proposition}[Invariance]
The $f$-sensitivity index $H_{u,f}$ of $Y$ for $\mathbf{X}_u$, $\emptyset\neq u\subseteq\{1,\cdots,N\}$, is invariant under smooth and uniquely invertible transformations (diffeomorphisms) of $Y$ and $\mathbf{X}_u$.
\label{p5}
\end{proposition}

\begin{proof}
For $\emptyset\neq u\subseteq\{1,\cdots,N\}$, let $\mathbf{X}_u'=\mathbf{g}_u(\mathbf{X}_u):=(g_{i_1}(\mathbf{X}_u),\cdots,g_{i_{|u|}}(\mathbf{X}_u))$ and $Y'=h(Y)$ be smooth and uniquely invertible, that is, diffeomorphic maps of random variables $\mathbf{X}_u=(X_{i_1},\cdots,X_{i_{|u|}})$ and $Y$.  From elementary probability theory, the probability densities of the transformed variables $\mathbf{X}_u'=(X_{i_1}',\cdots,X_{i_{|u|}}')$, $Y'$, and $(\mathbf{X}_u',Y')$ are
\begin{equation*}
\begin{array}{rcl}
f_{\mathbf{X}_u'}(\mathbf{x}_u')         & = &
f_{\mathbf{X}_u}(\mathbf{x}_u)
{\displaystyle \left| \det \mathbf{J}(\mathbf{x}_u) \right|^{-1} },      \\
f_{Y'}(\xi')                             & = & f_Y(\xi)
{\displaystyle \left| \dfrac{d \xi'}{d \xi} \right|^{-1} },  \\
f_{\mathbf{X}_u',Y'}(\mathbf{x}_u',\xi') & = &
f_{\mathbf{X}_u,Y}(\mathbf{x}_u,\xi)
{\displaystyle \left| \det \mathbf{J}(\mathbf{x}_u) \right|^{-1} }
{\displaystyle \left| \dfrac{d \xi'}{d \xi} \right|^{-1} },
\end{array}
\end{equation*}
respectively, where $\mathbf{J}(\mathbf{x}_u):=[\partial g_{i_p}/\partial x_{i_q}] \in \mathbb{R}^{|u|\times |u|}$, $p,q=1,\cdots,|u|$, is the Jacobian matrix of the transformation such that $\det \mathbf{J}(\mathbf{x}_u) \ne 0$ for any $\mathbf{x}_u \in \mathbb{R}^{|u|}$ and $d \xi'/d \xi=dh/d\xi \ne 0$.  Applying these relationships to the sensitivity index $H_{u,f}'$ of $Y'$ for $\mathbf{X}_u'$ defined in the last line of (\ref{5}) and noting $d\mathbf{x}_u'=|\det \mathbf{J}(\mathbf{x}_u)|d\mathbf{x}_u$ and $d\xi'=|d\xi'/d\xi|d\xi$ yields
\begin{equation*}
\begin{array}{rcl}
H_{u,f}' & = &
\int_{\mathbb{R}^{|u|}\times\mathbb{R}}
f\left(
\dfrac{f_{Y}'(\xi')f_{\mathbf{X}_u'}(\mathbf{x}_u')}
{f_{\mathbf{X}_u',Y'}(\mathbf{x}_u',\xi')}
\right)
f_{\mathbf{X}_{u}',Y'}(\mathbf{x}_u',\xi')d{\mathbf{x}_u'}d\xi' \\
         & = &
\int_{\mathbb{R}^{|u|}\times\mathbb{R}}
f\left(
\dfrac{f_{Y}(\xi) {\displaystyle \left| \dfrac{d \xi'}{d \xi} \right|^{-1} }
f_{\mathbf{X}_u}(\mathbf{x}_u) {\displaystyle \left| \det \mathbf{J}(\mathbf{x}_u) \right|^{-1} }
}
{f_{\mathbf{X}_u,Y}(\mathbf{x}_u,\xi)
{\displaystyle \left| \det \mathbf{J}(\mathbf{x}_u) \right|^{-1} }
{\displaystyle \left| \dfrac{d \xi'}{d \xi} \right|^{-1} }
}
\right)
f_{\mathbf{X}_{u},Y}(\mathbf{x}_u,\xi)
{\displaystyle \left| \det \mathbf{J}(\mathbf{x}_u) \right|^{-1} }
{\displaystyle \left| \dfrac{d \xi'}{d \xi} \right|^{-1} }     \\
         &   &
\times {\displaystyle \left| \det \mathbf{J}(\mathbf{x}_u) \right| }
{\displaystyle \left| \dfrac{d \xi'}{d \xi} \right| }
d{\mathbf{x}_u}d\xi                                            \\
         & = &
\int_{\mathbb{R}^{|u|}\times\mathbb{R}}
f\left(
\dfrac{f_{Y}(\xi)f_{\mathbf{X}_u}(\mathbf{x}_u)}
{f_{\mathbf{X}_u,Y}(\mathbf{x}_u,\xi)}
\right)
f_{\mathbf{X}_{u},Y}(\mathbf{x}_u,\xi)d{\mathbf{x}_u}d\xi       \\
         & = &
H_{u,f},
\end{array}
\end{equation*}
completing the proof.
\end{proof}

For a special case of $u=\{i\}$, $i=1,\cdots,N$, Corollary 4 of Borgonovo et al. \cite{borgonovo14} describes the monotonic invariance of a univariate sensitivity index derived from $\mathcal{L}_1$ norm or $f$-divergence. In contrast, Proposition \ref{p5} and its proof presented here are more general and different than those reported in the existing work \cite{borgonovo14}.

The invariance property of the $f$-sensitivity index described by Proposition \ref{p5} does not hold in general for the variance-based Sobol index \cite{borgonovo14}. The latter index is invariant only under affine transformations. Moreover, the $f$-sensitivity index, unlike the Sobol index, is applicable to random input following dependent probability distributions.

\begin{proposition}[Bounds for metric $f$-divergences]
Let $\emptyset\neq u, v\subset\{1,\cdots,N\}$ be two disjoint subsets such that $u \cap v = \emptyset$. For probability measures $P_Y$, $P_{Y|\mathbf{X}_u}$, and $P_{Y|\mathbf{X}_{u \cup v}}$, let $f$ be a select convex generating function, which produces metric $f$-divergences from $P_Y$ to $P_{Y|\mathbf{X}_{u \cup v}}$, from $P_Y$ to $P_{Y|\mathbf{X}_{u}}$, and from $P_{Y|\mathbf{X}_u}$ to $P_{Y|\mathbf{X}_{u \cup v}}$, satisfying the triangle inequality
\begin{equation*}
D_f\left( P_Y \parallel P_{Y|\mathbf{X}_{u \cup v}}\right) \le
D_f\left( P_Y \parallel P_{Y|\mathbf{X}_{u}}\right) +
D_f\left( P_{Y|\mathbf{X}_{u}} \parallel P_{Y|\mathbf{X}_{u \cup v}}\right).
\end{equation*}
Then
\begin{equation}
H_{u,f} \le H_{u \cup v,f} \le H_{u,f} + H_{v|u,f},
\label{sr2}
\end{equation}
where $H_{v|u,f}:=
\mathbb{E}_{\mathbf{X}_{u \cup v}}[D_f( P_{Y|\mathbf{X}_{u}} \parallel P_{Y|\mathbf{X}_{u \cup v}})]$ is the conditional sensitivity index of $Y|\mathbf{X}_{u}$ for $\mathbf{X}_{u \cup v}$.  Furthermore, if $\mathbf{X}_u$ and $\mathbf{X}_v$ are statistically independent, then
\begin{equation}
H_{u,f} \le H_{u \cup v,f} \le H_{u,f} + H_{v,f}.
\label{sr2b}
\end{equation}
\label{p6}
\end{proposition}

\begin{proof}
Applying the expectation operator $\mathbb{E}_{\mathbf{X}_{u \cup v}}$ on both sides of the triangle inequality yields

\begin{equation}
\begin{array}{cl}
      &  \int_{\mathbb{R}^{|u \cup v|}}
D_f\left( P_Y \parallel P_{Y|\mathbf{X}_{u \cup v}=\mathbf{x}_{u \cup v}}\right)
f_{\mathbf{X}_{u \cup v}}(\mathbf{x}_{u \cup v})d{\mathbf{x}_{u \cup v}}  \\
      &  ~~~~ \le \int_{\mathbb{R}^{|u \cup v|}}
D_f\left( P_Y \parallel P_{Y|\mathbf{X}_{u}=\mathbf{x}_{u}}\right)
f_{\mathbf{X}_{u \cup v}}(\mathbf{x}_{u \cup v})d{\mathbf{x}_{u \cup v}}  \\
      &  ~~~~~~~~+ \int_{\mathbb{R}^{|u \cup v|}}
D_f\left( P_{Y|\mathbf{X}_{u}=\mathbf{x}_{u}} \parallel P_{Y|\mathbf{X}_{u \cup v}=\mathbf{x}_{u \cup v}}\right)
f_{\mathbf{X}_{u \cup v}}(\mathbf{x}_{u \cup v})d{\mathbf{x}_{u \cup v}}.
\end{array}
\label{sr3}
\end{equation}
Since, for $u \cap v = \emptyset$, $D_f( P_Y \parallel P_{Y|\mathbf{X}_{u}=\mathbf{x}_{u}})$ does not depend on $\mathbf{x}_{v}$, the first integral on the right side of (\ref{sr3}) reduces to
\begin{equation*}
\begin{array}{cl}
      & \int_{\mathbb{R}^{|u \cup v|}}
D_f\left( P_Y \parallel P_{Y|\mathbf{X}_{u}=\mathbf{x}_{u}}\right)
f_{\mathbf{X}_{u \cup v}}(\mathbf{x}_{u \cup v})d{\mathbf{x}_{u \cup v}} \\
      & ~~~~ = \int_{\mathbb{R}^{|u|}}
D_f\left( P_Y \parallel P_{Y|\mathbf{X}_{u}=\mathbf{x}_{u}}\right)
\left(  \int_{\mathbb{R}^{|v|}}
f_{\mathbf{X}_{u \cup v}}(\mathbf{x}_{u \cup v})d{\mathbf{x}_{v}}
\right)
d{\mathbf{x}_{u}}   \\
      & ~~~~ = \int_{\mathbb{R}^{|u|}}
D_f\left( P_Y \parallel P_{Y|\mathbf{X}_{u}=\mathbf{x}_{u}}\right)
f_{\mathbf{X}_{u}}(\mathbf{x}_{u})d{\mathbf{x}_{u}}.
\end{array}
\end{equation*}
Therefore, (\ref{sr3}) becomes
\begin{equation}
\begin{array}{rcl}
      &  \int_{\mathbb{R}^{|u \cup v|}}
D_f\left( P_Y \parallel P_{Y|\mathbf{X}_{u \cup v}=\mathbf{x}_{u \cup v}}\right)
f_{\mathbf{X}_{u \cup v}}(\mathbf{x}_{u \cup v})d{\mathbf{x}_{u \cup v}}  \\
      &  ~~~~ \le \int_{\mathbb{R}^{|u|}}
D_f\left( P_Y \parallel P_{Y|\mathbf{X}_{u}=\mathbf{x}_{u}}\right)
f_{\mathbf{X}_{u}}(\mathbf{x}_{u})d{\mathbf{x}_{u}}   \\
      &  ~~~~~~~~~~~~~~~~~~~~~~~~~~~~~~+ \int_{\mathbb{R}^{|u \cup v|}}
D_f\left( P_{Y|\mathbf{X}_{u}=\mathbf{x}_{u}} \parallel P_{Y|\mathbf{X}_{u \cup v}=\mathbf{x}_{u \cup v}}\right)
f_{\mathbf{X}_{u \cup v}}(\mathbf{x}_{u \cup v})d{\mathbf{x}_{u \cup v}}.
\end{array}
\label{sr4}
\end{equation}
Recognizing the sensitivity indices $H_{u \cup v,f}$, $H_{u,f}$, and $H_{v|u,f}$ to be respectively the
integral on the left side, the first integral on the right side, and the second integral on the right side of (\ref{sr4}) produces the upper bound in (\ref{sr2}).  In addition, observe that the sensitivity index $H_{v|u,f}$ is non-negative, represents the contribution of the divergence from $P_{Y|\mathbf{X}_u}$ to $P_{Y|\mathbf{X}_{u \cup v}}$, and vanishes if and only if $Y$ and $\mathbf{X}_v$ are statistically independent.  Therefore, $H_{u \cup v,f}$ reaches the lower bound, which is $H_{u,f}$, if and only if $Y$ and $\mathbf{X}_v$ are statistically independent.

To obtain (\ref{sr2b}), use the last line of (\ref{5}) to write
\begin{equation}
H_{v|u,f} :=
\int_{\mathbb{R}^{|u \cup v|}\times\mathbb{R}}
f\left(
\dfrac{f_{Y|\mathbf{X}_u}(\xi|\mathbf{x}_u)f_{\mathbf{X}_{u \cup v}}(\mathbf{x}_{u \cup v})}
{f_{\mathbf{X}_{u \cup v,Y}}(\mathbf{x}_{u \cup v},\xi)}
\right)
f_{\mathbf{X}_{u \cup v,Y}}(\mathbf{x}_{u \cup v},\xi)d{\mathbf{x}_{u \cup v}}d\xi,
\label{sr4b}
\end{equation}
where, by invoking the statistical independence between $\mathbf{X}_u$ and $\mathbf{X}_v$, the numerator and denominator of the argument of $f$ become
\begin{equation}
\begin{array}{rcl}
f_{Y|\mathbf{X}_u}(\xi|\mathbf{x}_u)f_{\mathbf{X}_{u \cup v}}(\mathbf{x}_{u \cup v})
        & = &
f_{Y|\mathbf{X}_u}(\xi|\mathbf{x}_u)f_{\mathbf{X}_{u}}(\mathbf{x}_{u})f_{\mathbf{X}_{v}}(\mathbf{x}_{v}) \\
        & = &
f_{\mathbf{X}_u,Y}(\mathbf{x}_{u},\xi)f_{\mathbf{X}_{v}}(\mathbf{x}_{v})  \\
        & = &
f_{\mathbf{X}_u|Y}(\mathbf{x}_{u}|\xi)f_Y(\xi)f_{\mathbf{X}_{v}}(\mathbf{x}_{v})
\end{array}
\label{sr4c}
\end{equation}
and
\begin{equation}
f_{\mathbf{X}_{u \cup v,Y}}(\mathbf{x}_{u \cup v},\xi) =
f_{\mathbf{X}_u|(\mathbf{X}_v,Y)}(\mathbf{x}_{u}|(\mathbf{x}_v,\xi))f_{\mathbf{X}_{v,Y}}(\mathbf{x}_{v},\xi) =
f_{\mathbf{X}_u|Y}(\mathbf{x}_{u}|\xi)f_{\mathbf{X}_{v,Y}}(\mathbf{x}_{v},\xi),
\label{sr4d}
\end{equation}
respectively.  Applying (\ref{sr4c}) and (\ref{sr4d}) to (\ref{sr4b}) results in
\begin{equation*}
\begin{array}{rcl}
H_{v|u,f}
        & = &
\int_{\mathbb{R}^{|u \cup v|}\times \mathbb{R}}
f\left(
\dfrac{f_Y(\xi)f_{\mathbf{X}_{v}}(\mathbf{x}_{v})}
{f_{\mathbf{X}_{v,Y}}(\mathbf{x}_{v},\xi)}
\right)
f_{\mathbf{X}_{u \cup v,Y}}(\mathbf{x}_{u \cup v},\xi)d{\mathbf{x}_{u \cup v}}d\xi \\
        & = &
\int_{\mathbb{R}^{|v|}\times \mathbb{R}}
f\left(
\dfrac{f_Y(\xi)f_{\mathbf{X}_{v}}(\mathbf{x}_{v})}
{f_{\mathbf{X}_{v,Y}}(\mathbf{x}_{v},\xi)}
\right)
\left(
\int_{\mathbb{R}^{|u|}}
f_{\mathbf{X}_{u \cup v,Y}}(\mathbf{x}_{u \cup v},\xi)d{\mathbf{x}_{u}}
\right)
d{\mathbf{x}_{v}}d\xi \\
        & = &
\int_{\mathbb{R}^{|v|}\times \mathbb{R}}
f\left(
\dfrac{f_Y(\xi)f_{\mathbf{X}_{v}}(\mathbf{x}_{v})}
{f_{\mathbf{X}_{v,Y}}(\mathbf{x}_{v},\xi)}
\right)
f_{\mathbf{X}_{v,Y}}(\mathbf{x}_{v},\xi) d{\mathbf{x}_{v}}d\xi \\
        & = &
H_{v,f},
\end{array}
\end{equation*}
which transforms (\ref{sr2}) to (\ref{sr2b}) and hence completes the proof.
\end{proof}

As a special case, consider again $u=\{i\}$ and $v=\{j\}$, where $i,j = 1,\cdots,N$, $i \ne j$.  Then, according to Proposition \ref{p6}, applicable to sensitivity indices rooted in metric $f$-divergences only,
\begin{equation*}
H_{\{i\},f} \le H_{\{i,j\},f} \le H_{\{i\},f}+H_{\{j|i\},f},
\end{equation*}
which states the following: if $Y$ depends on $X_j$, then the contribution of $X_j$ to the sensitivity of $Y$ for $(X_i,X_j)$ increases from $H_{\{i\},f}$, but is limited by the residual term $H_{\{j|i\},f}$.  If $Y$ and $X_j$ are statistically independent, then $H_{\{j|i\},f}$ vanishes, resulting in $H_{\{i,j\},f} = H_{\{i\},f}$.  This agrees with Proposition \ref{p4}, which, however, is valid whether or not the underlying $f$-divergence is a metric.  In addition, if $X_i$ and $X_j$ are statistically independent, then $H_{\{j|i\},f}=H_{\{j\},f}$, yielding  $H_{\{i\},f} \le H_{\{i,j\},f} \le H_{\{i\},f}+H_{\{j\},f}$. Borgonovo \cite{borgonovo07} derived the same bounds for a special case when the sensitivity index stems from the total variational distance.  Proposition \ref{p6}, by contrast, is a general result and applicable to sensitivity indices emanating from all metric $f$-divergences.

\subsection{Special cases}
A plethora of $f$-sensitivity indices are possible by appropriately selecting the convex function $f$ in (\ref{4b}) or (\ref{5}).  Listed in Table \ref{table1} are ten such sensitivity indices derived from the forward and reversed Kullback-Leibler divergences, total variational distance, Hellinger distance, Pearson $\chi^2$ divergence, Neyman $\chi^2$ divergence, $\alpha$ divergence, Vajda $\chi^\alpha$ divergence, Jeffreys distance, and triangular discrimination in (\ref{3a}) through (\ref{3j}).  Three prominent sensitivity indices, for example, the mutual information \cite{cover91}
\begin{equation*}
I_u:=
\int_{\mathbb{R}^{|u|}\times\mathbb{R}}
\ln \left[
\dfrac{f_{\mathbf{X}_u,Y}(\mathbf{x}_u,\xi)}{f_{Y}(\xi) f_{\mathbf{X}_u}(\mathbf{x}_u)}
\right]
f_{\mathbf{X}_u,Y}(\mathbf{x}_u,\xi)
d{\mathbf{x}_u}d\xi =:H_{u,KL'}
\end{equation*}
between $\mathbf{X}_u$ and $Y$, the squared-loss mutual information \cite{suzuki09}
\begin{equation*}
\begin{array}{rcl}
I'_u & := &
\int_{\mathbb{R}^{|u|}\times\mathbb{R}}
\left[
\dfrac{f_{\mathbf{X}_u,Y}(\mathbf{x}_u,\xi)}{f_{Y}(\xi) f_{\mathbf{X}_u}(\mathbf{x}_u)} - 1
\right]^2  f_{Y}(\xi) f_{\mathbf{X}_u}(\mathbf{x}_u)  d{\mathbf{x}_u}d\xi  \\
     & =  &
\int_{\mathbb{R}^{|u|}\times\mathbb{R}}
\dfrac{f_{\mathbf{X}_u,Y}(\mathbf{x}_u,\xi)}{f_{Y}(\xi) f_{\mathbf{X}_u}(\mathbf{x}_u)}
\left[1 - \left\{
\dfrac{f_{Y}(\xi) f_{\mathbf{X}_u}(\mathbf{x}_u)}{f_{\mathbf{X}_u,Y}(\mathbf{x}_u,\xi)}
\right\}^2
\right]
f_{\mathbf{X}_u,Y}(\mathbf{x}_u,\xi)
d{\mathbf{x}_u}d\xi                                                     \\
     & =: &
H_{u,N}
\end{array}
\end{equation*}
between $\mathbf{X}_u$ and $Y$, and Borgonovo's importance measure \cite{borgonovo07}
\begin{equation*}
\begin{array}{rcl}
\delta_u & := &
\dfrac{1}{2}\int_{\mathbb{R}^{|u|}\times\mathbb{R}}
\left|
f_{Y}(\xi) f_{\mathbf{X}_u}(\mathbf{x}_u)-f_{\mathbf{X}_u,Y}(\mathbf{x}_u,\xi)
\right|
d{\mathbf{x}_u}d\xi                                                      \\
         & =  &
\dfrac{1}{2}\int_{\mathbb{R}^{|u|}\times\mathbb{R}}
\left|
\dfrac{f_{Y}(\xi) f_{\mathbf{X}_u}(\mathbf{x}_u)}{f_{\mathbf{X}_u,Y}(\mathbf{x}_u,\xi)}-1
\right|
f_{\mathbf{X}_u,Y}(\mathbf{x}_u,\xi)
d{\mathbf{x}_u}d\xi                                                       \\
         & =: &
\dfrac{1}{2}H_{u,TV}
\end{array}
\end{equation*}
of $\mathbf{X}_u$ on $Y$, are rooted in reversed Kullback-Leibler, Neyman, and total variational divergences or distances, respectively. Indeed, many previously used sensitivity or importance measures are special cases of the $f$-sensitivity index derived from the $f$-divergence.
\begin{table}
\caption{Ten special cases of the $f$-sensitivity index }
\raggedright{}{\footnotesize{}}%
\begin{centering}
\begin{tabular}{lll}
\hline
{\footnotesize{$f$-}}\textbf{\footnotesize{divergence}} & \textbf{\footnotesize{$f(t)$}}$^{(\mathrm{a})}$ & \textbf{\footnotesize{Sensitivity index}}\tabularnewline
\hline
 &  & \tabularnewline
{\footnotesize{$\begin{array}{l}
\!\!\!\!\mathrm{Forward\: Kullback-}\\
\!\!\!\!\mathrm{Leibler\: divergence}
\end{array}$}} & {\footnotesize{$t\ln t$}} & {\footnotesize{$\begin{array}{rl}\!\!\!\!
H_{u,KL}:= & \int_{\mathbb{R}^{|u|}\times\mathbb{R}}\dfrac{f_{Y}(\xi)f_{\mathbf{X}_{u}}(\mathbf{x}_{u})}{f_{\mathbf{X}_{u},Y}(\mathbf{x}_{u},\xi)}\ln\left[\dfrac{f_{Y}(\xi)f_{\mathbf{X}_{u}}(\mathbf{x}_{u})}{f_{\mathbf{X}_{u},Y}(\mathbf{x}_{u},\xi)}\right]\\
 & \times f_{\mathbf{X}_{u},Y}(\mathbf{x}_{u},\xi)d\mathbf{x}_{u}d\xi
\end{array}$}}\tabularnewline
\noalign{\vskip0.2cm}
{\footnotesize{$\begin{array}{l}
\!\!\!\!\mathrm{Reversed\: Kullback-}\\
\!\!\!\!\mathrm{Leibler\: divergence}
\end{array}$}} & {\footnotesize{$-\ln t$}} & {\footnotesize{$\begin{array}{rl}\!\!\!\!
H_{u,KL'}:= & \int_{\mathbb{R}^{|u|}\times\mathbb{R}}\ln\left[\dfrac{f_{\mathbf{X}_{u},Y}(\mathbf{x}_{u},\xi)}{f_{Y}(\xi)f_{\mathbf{X}_{u}}(\mathbf{x}_{u})}\right]f_{\mathbf{X}_{u},Y}(\mathbf{x}_{u},\xi)d\mathbf{x}_{u}d\xi\end{array}$}}\tabularnewline
\noalign{\vskip0.2cm}
{\footnotesize{$\begin{array}{l}
\!\!\!\!\mathrm{Total\: variational-}\\
\!\!\!\!\mathrm{distance}
\end{array}$}} & {\footnotesize{$\left|t-1\right|$}} & {\footnotesize{$\begin{array}{rl}\!\!\!\!
H_{u,TV}:= & \int_{\mathbb{R}^{|u|}\times\mathbb{R}}\left|\dfrac{f_{Y}(\xi)f_{\mathbf{X}_{u}}(\mathbf{x}_{u})}{f_{\mathbf{X}_{u},Y}(\mathbf{x}_{u},\xi)}-1\right|f_{\mathbf{X}_{u},Y}(\mathbf{x}_{u},\xi)d\mathbf{x}_{u}d\xi\end{array}$}}\tabularnewline
\noalign{\vskip0.2cm}
{\footnotesize{$\begin{array}{l}
\!\!\!\!\mathrm{Hellinger}\\
\!\!\!\!\mathrm{distance}
\end{array}$}} & {\footnotesize{$\left(\sqrt{t}-1\right)^{2}$}} & {\footnotesize{$\begin{array}{rl}\!\!\!\!
H_{u,H}:= & \int_{\mathbb{R}^{|u|}\times\mathbb{R}}\left[\sqrt{\dfrac{f_{Y}(\xi)f_{\mathbf{X}_{u}}(\mathbf{x}_{u})}{f_{\mathbf{X}_{u},Y}(\mathbf{x}_{u},\xi)}}-1\right]^{2}f_{\mathbf{X}_{u},Y}(\mathbf{x}_{u},\xi)d\mathbf{x}_{u}d\xi\end{array}$}}\tabularnewline
\noalign{\vskip0.2cm}
{\footnotesize{$\begin{array}{l}
\!\!\!\!\mathrm{Pearson\:\chi^{2}}\\
\!\!\!\!\mathrm{divergence}
\end{array}$}} & {\footnotesize{$t^{2}-1$}} & {\footnotesize{$\begin{array}{rl}\!\!\!\!
H_{u,P}:= & \int_{\mathbb{R}^{|u|}\times\mathbb{R}}\left[\left\{ \dfrac{f_{Y}(\xi)f_{\mathbf{X}_{u}}(\mathbf{x}_{u})}{f_{\mathbf{X}_{u},Y}(\mathbf{x}_{u},\xi)}\right\} ^{2}-1\right]f_{\mathbf{X}_{u},Y}(\mathbf{x}_{u},\xi)d\mathbf{x}_{u}d\xi\end{array}$}}\tabularnewline
\noalign{\vskip0.2cm}
{\footnotesize{$\begin{array}{l}
\!\!\!\!\mathrm{Neyman\:\chi^{2}}\\
\!\!\!\!\mathrm{divergence}
\end{array}$}} & {\footnotesize{$\dfrac{1-t^{2}}{t}$}} & {\footnotesize{$\begin{array}{rl}\!\!\!\!
H_{u,N}:= & \int_{\mathbb{R}^{|u|}\times\mathbb{R}}\dfrac{f_{\mathbf{X}_{u},Y}(\mathbf{x}_{u},\xi)}{f_{Y}(\xi)f_{\mathbf{X}_{u}}(\mathbf{x}_{u})}\left[1-\left\{ \dfrac{f_{Y}(\xi)f_{\mathbf{X}_{u}}(\mathbf{x}_{u})}{f_{\mathbf{X}_{u},Y}(\mathbf{x}_{u},\xi)}\right\} ^{2}\right]\\
 & \times f_{\mathbf{X}_{u},Y}(\mathbf{x}_{u},\xi)d\mathbf{x}_{u}d\xi
\end{array}$}}\tabularnewline
\noalign{\vskip0.2cm}
{\footnotesize{$\alpha-\mathrm{divergence}$$^{(\mathrm{b})}$}} & {\footnotesize{$\dfrac{4\left[1-t^{(1-\alpha)/2}\right]}{1-\alpha^{2}}$}} & {\footnotesize{$\begin{array}{rl}\!\!\!\!
H_{u,\alpha}:= & \dfrac{4}{1-\alpha^{2}}\int_{\mathbb{R}^{|u|}\times\mathbb{R}}\left[1-\left\{ \dfrac{f_{Y}(\xi)f_{\mathbf{X}_{u}}(\mathbf{x}_{u})}{f_{\mathbf{X}_{u},Y}(\mathbf{x}_{u},\xi)}\right\} ^{(1-\alpha)/2}\right]\\
 & \times f_{\mathbf{X}_{u},Y}(\mathbf{x}_{u},\xi)d\mathbf{x}_{u}d\xi
\end{array}$}}\tabularnewline
\noalign{\vskip0.2cm}
{\footnotesize{$\begin{array}{l}
\!\!\!\!\mathrm{Vajda\:\chi^\alpha}\\
\!\!\!\!\mathrm{divergence}^{(\mathrm{c})}
\end{array}$}} & {\footnotesize{$|t-1|^\alpha$}} & {\footnotesize{$\begin{array}{rl}\!\!\!\!
H_{u,V}:= & \int_{\mathbb{R}^{|u|}\times\mathbb{R}}\left|\dfrac{f_{Y}(\xi)f_{\mathbf{X}_{u}}(\mathbf{x}_{u})}{f_{\mathbf{X}_{u},Y}(\mathbf{x}_{u},\xi)}-1\right|^\alpha f_{\mathbf{X}_{u},Y}(\mathbf{x}_{u},\xi)d\mathbf{x}_{u}d\xi\end{array}$}}\tabularnewline
\noalign{\vskip0.2cm}
{\footnotesize{$\begin{array}{l}
\!\!\!\!\mathrm{Jeffreys}\\
\!\!\!\!\mathrm{distance}
\end{array}$}} & {\footnotesize{$(t-1)\ln t$}} & {\footnotesize{$\begin{array}{rl}\!\!\!\!
H_{u,J}:= & \int_{\mathbb{R}^{|u|}\times\mathbb{R}}\left[\dfrac{f_{Y}(\xi)f_{\mathbf{X}_{u}}(\mathbf{x}_{u})}{f_{\mathbf{X}_{u},Y}(\mathbf{x}_{u},\xi)}-1\right]\ln\left[\dfrac{f_{Y}(\xi)f_{\mathbf{X}_{u}}(\mathbf{x}_{u})}{f_{\mathbf{X}_{u},Y}(\mathbf{x}_{u},\xi)}\right]\\
 & \times f_{\mathbf{X}_{u},Y}(\mathbf{x}_{u},\xi)d\mathbf{x}_{u}d\xi
\end{array}$}}\tabularnewline
\noalign{\vskip0.2cm}
{\footnotesize{$\begin{array}{l}
\!\!\!\!\mathrm{Triangular}\\
\!\!\!\!\mathrm{discrimination}
\end{array}$}} & {\footnotesize{$\dfrac{(t-1)^{2}}{t+1}$}} & {\footnotesize{$\begin{array}{rl}\!\!\!\!
H_{u,\triangle}:= & \int_{\mathbb{R}^{|u|}\times\mathbb{R}}\dfrac{\left[\dfrac{f_{Y}(\xi)f_{\mathbf{X}_{u}}(\mathbf{x}_{u})}{f_{\mathbf{X}_{u},Y}(\mathbf{x}_{u},\xi)}-1\right]^{2}}{\dfrac{f_{Y}(\xi)f_{\mathbf{X}_{u}}(\mathbf{x}_{u})}{f_{\mathbf{X}_{u},Y}(\mathbf{x}_{u},\xi)}+1}f_{\mathbf{X}_{u},Y}(\mathbf{x}_{u},\xi)d\mathbf{x}_{u}d\xi\end{array}$}}\tabularnewline
\noalign{\vskip0.2cm}
 &  & \tabularnewline
\hline
\end{tabular}
\end{centering}
{\footnotesize \par}

{\scriptsize{(a) All generating functions have been normalized, that is, $f(1)=0$.}}{\scriptsize \par}

{\scriptsize{(b) $\alpha \ne \pm 1$.}}{\scriptsize \par}

{\scriptsize{(c) $\alpha \ge 1$.  If $\alpha=1$, then $H_{u,V}=H_{u,TV}$.}}{\scriptsize \par}
\label{table1}
\end{table}

\subsection{Univariate index}
For practical applications, one may be interested in performing sensitivity analysis pertaining to individual random variables only. Setting $u=\{i\}$, where $i=1,\cdots,N$, in (\ref{4}) through (\ref{5}), the corresponding $f$-divergence and univariate $f$-sensitivity index of $Y$ with respect to a variable $X_i$ are
\begin{equation}
D_f \left( P_Y \parallel P_{Y|X_i}\right):=
\int_{\mathbb{R}}
f\left(\dfrac{f_{Y}(\xi)}{f_{Y|X_i}(\xi|x_i)}\right) f_{Y|X_i}(\xi|x_i)d\xi
\label{6}
\end{equation}
and
\begin{equation}
H_{\{i\},f}:=
\mathbb{E}_{X_i} \left[ D_f\left( P_Y \parallel P_{Y|X_i}\right) \right]=
\int_{\mathbb{R}\times\mathbb{R}}
f\left(
\dfrac{f_{Y}(\xi)f_{X_i}(x_i)}
{f_{X_i,Y}(x_i,\xi)}
\right)
f_{X_i,Y}(x_i,\xi)d{x_i}d\xi,
\label{7}
\end{equation}
respectively. Again, depending on the choice of the convex function $f$ in (\ref{6}) and (\ref{7}), many univariate $f$-sensitivity indices from Table 1 with $u=\{i\}$ can be generated.  For instance, the univariate $f$-sensitivity indices derived from the total variational distance, reversed Kullback-Leibler divergence, Neyman $\chi^2$ divergence, and Hellinger distance are
\begin{subequations}
\begin{equation}
H_{\{i\},TV}:=
\int_{\mathbb{R}\times\mathbb{R}}
\left|
\dfrac{f_{Y}(\xi) f_{X_i}(x_i)}{f_{X_i,Y}(x_i,\xi)}-1
\right|
f_{X_i,Y}(x_i,\xi)
d{x_i}d\xi,
\label{8a}
\end{equation}
\begin{equation}
H_{\{i\},KL'}:=
\int_{\mathbb{R}\times\mathbb{R}}
\ln \left[
\dfrac{f_{X_i,Y}(x_i,\xi)}{f_Y(\xi) f_{X_i}(x_i)}
\right]
f_{X_i,Y}(x_i,\xi)
d{x_i}d\xi,
\label{8b}
\end{equation}
\begin{equation}
H_{\{i\},N}:=
\int_{\mathbb{R}\times\mathbb{R}}
\dfrac{f_{X_i,Y}(x_i,\xi)}{f_{Y}(\xi) f_{X_i}(x_i)}
\left[1 - \left\{
\dfrac{f_{Y}(\xi) f_{X_i}(x_i)}{f_{X_i,Y}(x_i,\xi)}
\right\}^2
\right]
f_{X_i,Y}(x_i,\xi)
d{x_i}d\xi,
\label{8c}
\end{equation}
and
\begin{equation}
H_{\{i\},H}:=
\int_{\mathbb{R}\times\mathbb{R}}
\left[
\sqrt{
\dfrac{f_{Y}(\xi) f_{X_i}(x_i)}{f_{X_i,Y}(x_i,\xi)}
} - 1
\right]^2
f_{X_i,Y}(x_i,\xi)
d{x_i}d\xi,
\label{8d}
\end{equation}
\end{subequations}
respectively.  A numerical evaluation of the univariate indices defined in (\ref{8a}) through (\ref{8d}) will be discussed in section 4.

\subsection{Generalization}
The $f$-divergence and $f$-sensitivity index in (\ref{4}) through (\ref{5}) easily extend for an $M$-dimensional output vector $\mathbf{Y}=(y_{1}(\mathbf{X}),\cdots,y_{M}(\mathbf{X}))$, where $M \in \mathbb{N}$.  In this case, the $f$-divergence between two multivariate probability measures $P_{\mathbf{Y}}$ and $P_{\mathbf{Y}|\mathbf{X}_u}$ is defined by
\begin{equation*}
D_f\left( P_{\mathbf{Y}} \parallel P_{\mathbf{Y}|\mathbf{X}_u}\right):=
\int_{\mathbb{R}^M}f\left(\dfrac{f_{\mathbf{Y}}(\boldsymbol{\xi})}
{f_{\mathbf{Y}|\mathbf{X}_{u}}(\boldsymbol{\xi}|\mathbf{x}_u)}\right) f_{\mathbf{Y}|\mathbf{X}_{u}}(\boldsymbol{\xi}|\mathbf{x}_u)d\boldsymbol{\xi},
\end{equation*}
leading to the $f$-sensitivity index
\begin{equation*}
H_{u,f} := \mathbb{E}_{\mathbf{X}_u} \left[ D_f\left( P_{\mathbf{Y}} \parallel P_{\mathbf{Y}|\mathbf{X}_u}\right) \right] =
\int_{\mathbb{R}^{|u|}\times\mathbb{R}^M}
f\left(
\dfrac{f_{\mathbf{Y}}(\boldsymbol{\xi})f_{\mathbf{X}_u}(\mathbf{x}_u)}
{f_{\mathbf{X}_u,\mathbf{Y}}(\mathbf{x}_u,\boldsymbol{\xi})}
\right)
f_{\mathbf{X}_{u},\mathbf{Y}}(\mathbf{x}_u,\boldsymbol{\xi})d{\mathbf{x}_u}d\boldsymbol{\xi}.
\end{equation*}
Here, $f_{\mathbf{Y}}(\boldsymbol{\xi})$, $f_{\mathbf{Y}|\mathbf{X}_{u}}(\boldsymbol{\xi}|\mathbf{x}_u)$, and $f_{\mathbf{X}_{u},\mathbf{Y}}(\mathbf{x}_u,\boldsymbol{\xi})$ are the probability density functions of $\mathbf{Y}$, $\mathbf{Y}|\mathbf{X}_u$, and $(\mathbf{X}_u,\mathbf{Y})$, respectively.  However, calculating such a multivariate index for an arbitrarily large $|u|$ or $M$ is highly nontrivial and is not addressed here.

\section{Global sensitivity analysis}
While the formulation of the $f$-sensitivity index is not unduly difficult, its calculation is a different matter and by and large a formidable task. If the convex function $f$ is already selected, resulting in a specific sensitivity index, then one can exploit the functional form of $f$ to devise accurate and efficient methods of calculation. This is exemplified in the works of Borgonovo \cite{borgonovo07}, Liu and Homma \cite{liu09}, and Wei, Lu, and Yuan \cite{wei13}, presenting several estimation procedures for calculating Borgonovo's importance measure.  More recent works involve surrogate approximations \cite{borgonovo12} and improved estimators \cite{castaings12,plischke13} for the same index.  Here, the author delves into calculating the $f$-sensitivity index derived from a general convex function $f$, so that the methods proposed are applicable to a host of sensitivity indices.

In reference to the last line of (\ref{5}), the $f$-sensitivity index $H_{u,f}$ is exactly calculated only when both the integrand function $f$, which depends on the probability densities $f_{Y}(\xi)$ and $f_{\mathbf{X}_u,Y}(\mathbf{x}_u,\xi)$, and the subsequent $(|u|+1)$-dimensional integral with respect to the probability measure of $(\mathbf{X}_u,Y)$ are exactly determined.  Depending on what can be determined exactly or not, approximate methods must be used to estimate $H_{u,f}$.  Three methods depending on specific scenarios expected in performing sensitivity analysis are proposed.

\subsection{The Monte Carlo or MC method}
Suppose that $y:\mathbb{R}^N \to \mathbb{R}$ is a mapping simple enough to produce exactly the probability densities $f_{Y}(\xi)$ and $f_{\mathbf{X}_u,Y}(\mathbf{x}_u,\xi)$.  For instance, if $y$ is an affine map of a Gaussian random vector $\mathbf{X}$, then $f_{Y}(\xi)$ and $f_{\mathbf{X}_u,Y}(\mathbf{x}_u,\xi)$ are both Gaussian density functions. However, since $f$ is unspecified, performing analytical integration still may not be possible.  The Monte Carlo (MC) method solves this problem by approximating the multi-dimensional integral by random sampling as follows.

Given a sample size $L \in \mathbb{N}$, let $\{\mathbf{x}^{(l)},\xi^{(l)}\}_{l=1,\cdots,L}$ be the set of input-output sample pairs of $(\mathbf{X},Y)$.  The input samples can be generated from the known probability density function of $\mathbf{X}$, whereas the output sample can be obtained from the known input-output transformation.  Then the index $H_{u,f}$, expressed in the last line of (\ref{5}), is approximated by the Monte Carlo (MC) estimator
\begin{equation}
\hat{H}_{u,f}^{(L)}: = \dfrac{1}{L}\sum_{l=1}^L
f\left(
\dfrac{f_{Y}(\xi^{(l)})f_{\mathbf{X}_u}(\mathbf{x}_u^{(l)})}
{f_{\mathbf{X}_u,Y}(\mathbf{x}_u^{(l)},\xi^{(l)})}
\right).
\label{9}
\end{equation}
If $f$ is square integrable with respect to $f_{\mathbf{X}_u,Y}(\mathbf{x}_u,\xi)d{\mathbf{x}_u}d\xi$, then according to the strong law of large numbers, $\hat{H}_{u,f}^{(L)} \to H_{u,f}$ as $L \to \infty$ with probability one.  Furthermore, the estimation is unbiased, which means that the average of $\hat{H}_{u,f}^{(L)}$ is exactly $H_{u,f}$.

\subsection{The KDE-MC method}
For an arbitrary response function or probability distribution of random input, obtaining exact probability densities of responses is wishful thinking. In reality, $y(\mathbf{X})$ may be highly nonlinear, where the input variables $\mathbf{X}$ may not follow classical probability distributions.  In which case, a new layer of approximation in determining the probability densities must be dealt with. A straightforward, practical approach entails employing kernel density estimation (KDE) \cite{parzen62,rosenblatt56} of the probability densities and then applying Monte Carlo integration.  The resulting method is referred to as the KDE-MC method.

As in the MC method, assume that for $L \in \mathbb{N}$, the input-output sample pairs $\{\mathbf{x}^{(l)},\xi^{(l)}\}_{l=1,\cdots,L}$ of $(\mathbf{X},Y)$ are available.  Using these samples, the KDEs of the probability densities $f_{Y}(\xi)$ and $f_{\mathbf{X}_u,Y}(\mathbf{x}_u,\xi)$ take the form
\begin{subequations}
\begin{equation}
\bar{f}_{Y}(\xi)=\dfrac{1}{L h_Y}\sum_{l=1}^L
K_Y \left( \dfrac{\xi-\xi^{(l)}}{h_Y} \right)
\label{10a}
\end{equation}
and
\begin{equation}
\bar{f}_{\mathbf{X}_u,Y}(\mathbf{x}_u,\xi)=
\dfrac{1}{L h_Y {\displaystyle \prod_{i \in u} h_{X_i}}}\sum_{l=1}^L
K_Y \left( \dfrac{\xi-\xi^{(l)}}{h_Y} \right)
{\displaystyle \prod_{i \in u}
K_{X_i} \left( \dfrac{x_i-x_i^{(l)}}{h_{X_i}} \right)
},
\label{10b}
\end{equation}
\end{subequations}
in which $K_{X_i}:\mathbb{R} \to \mathbb{R}$, $i \in u$, and $K_Y:\mathbb{R} \to \mathbb{R}$ are univariate kernel functions, whereas $h_{X_i}$, $i \in u$, and $h_Y$ are smoothing parameters called the bandwidths. Substituting the probability densities in (\ref{9}) with their KDEs in (\ref{10a}) and (\ref{10b}) results in the KDE-MC estimator
\begin{equation}
\bar{H}_{u,f}^{(L)} := \dfrac{1}{L}\sum_{l=1}^L
f\left(
\dfrac{\bar{f}_{Y}(\xi^{(l)}){f}_{\mathbf{X}_u}(\mathbf{x}_u^{(l)})}
{\bar{f}_{\mathbf{X}_u,Y}(\mathbf{x}_u^{(l)},\xi^{(l)})}
\right)
\label{11}
\end{equation}
of $H_{u,f}$.  It is well known that the asymptotic mean-squared error committed by the KDE increases with the bandwidth size but decreases in the product of bandwidth and sample sizes.  Therefore, for the KDE error to decline as $L \to \infty$, the bandwidth must decrease, but not at a rate faster than the sample size.  This is sufficient to establish pointwise convergence of the KDE.  In which case, the $\bar{H}_{u,f}^{(L)}$ in (\ref{11}) should furnish a good approximation of $H_{u,f}$ if $L$ is sufficiently large.

The MC and KDE-MC methods presented so far both require $L$ evaluations of the response function $y$ to estimate the $f$-sensitivity index.  Therefore, their computational costs, measured in terms of numbers of function evaluations alone, are the same. If the sample size $L$, concomitant with a required accuracy in estimating the sensitivity index, is very large, say, in the order of millions, then both methods will be restricted to problems or functions that are inexpensive to evaluate.  In a practical setting, however, the response function $y$ is often determined via time-consuming finite-element analysis (FEA) or similar numerical calculations.  In which case, an arbitrarily large sample size is no longer viable, and hence alternative routes to estimating the probability densities in (\ref{9}) should be charted.

\subsection{The PDD-KDE-MC method}
When the response of a complex system is expensive to evaluate, mathematically rigorous yet computationally efficient surrogate approximations can be applied for sensitivity analysis.  Modern surrogate approximations include polynomial chaos expansion (PCE) \cite{wiener38}, stochastic collocation \cite{babuska07}, and PDD \cite{rahman09,rahman08}, to name a few. All of these methods, commonly used for uncertainty quantification of complex systems, are known to offer significant computational advantages over crude Monte Carlo simulation. However, for truly high-dimensional problems, the PCE and collocation methods require astronomically large numbers of terms or coefficients, succumbing to the curse of dimensionality. The PDD, derived from the ANOVA decomposition, also reduces the computational effort, but more importantly, it deflates the curse of dimensionality to an extent determined by the degree of interaction among input variables \cite{rahman09,rahman08}.  This was the principal motivation for coupling PDD with the KDE-MC method described in the preceding subsection. The end product is an extended version, which is referred to as the PDD-KDE-MC method in this paper.

\subsubsection{PDD approximation}
Let $\mathcal{L}_{2}(\Omega,\mathcal{F},P)$
represent a Hilbert space of square-integrable functions $y$ with
respect to the generic probability measure $f_{\mathbf{X}}(\mathbf{x})d\mathbf{x}$
supported on $\mathbb{R}^{N}$. For a given $\emptyset\ne u\subseteq\{1,\cdots,N\}$, let $\{\psi_{u\mathbf{j}_{|u|}}(\mathbf{X}_{u});\:\mathbf{j}_{|u|}\in\mathbb{N}_{0}^{|u|}\}$, where $\mathbf{j}_{|u|}=(j_{1},\cdots,j_{|u|})\in\mathbb{N}_{0}^{|u|}$
is a $|u|$-dimensional multi-index, represent a set of multivariate orthonormal polynomials that is consistent with the probability measure of $\mathbf{X}_{u}$. The PDD of the function $y$ represents a finite, hierarchical expansion \cite{rahman09,rahman08},
\begin{equation}
y(\mathbf{X})=y_{\emptyset}+{\displaystyle \sum_{\emptyset\ne u\subseteq\{1,\cdots,N\}}}\:\sum_{{\textstyle {\mathbf{j}_{|u|}\in\mathbb{N}_{0}^{|u|}\atop j_{1},\cdots,j_{|u|}\neq0}}}C_{u\mathbf{j}_{|u|}}\psi_{u\mathbf{j}_{|u|}}(\mathbf{X}_{u}),
\label{12}
\end{equation}
in terms of random multivariate orthonormal polynomials
of input variables with increasing dimensions, where
\begin{subequations}
\begin{equation}
y_{\emptyset} := \int_{\mathbb{R}^{N}}y(\mathbf{x})f_{\mathbf{X}}(\mathbf{x})d\mathbf{x}
\label{13a}
\end{equation}
and
\begin{equation}
C_{u\mathbf{j}_{|u|}}:=\int_{\mathbb{R}^{N}}y(\mathbf{x})\psi_{u\mathbf{j}_{|u|}}(\mathbf{\mathbf{x}}_{u})f_{\mathbf{X}}(\mathbf{x})d\mathbf{x},\;\emptyset\ne u\subseteq\{1,\cdots,N\},\;\mathbf{j}_{|u|}\in\mathbb{N}_{0}^{|u|},
\label{13b}
\end{equation}
\end{subequations}
are various expansion coefficients.  Note that the summation in (\ref{12}) precludes $j_{1},\cdots,j_{|u|}=0$, that is, the individual degree of each variable $X_{i}$ in $\psi_{u\mathbf{j}_{|u|}}$, $i\in u$, cannot be \emph{zero} since $\psi_{u\mathbf{j}_{|u|}}$ is a strictly $|u|$-variate function and has a \emph{zero} mean \cite{rahman09,rahman08}.  The expression of $y_{\emptyset}$ in (\ref{13a}) is valid whether or not the random input $\mathbf{X}$ comprises independent or dependent variables.  However, the expression of $C_{u\mathbf{j}_{|u|}}$ in (\ref{13b}) is applicable only when the input variables are independent.  Although (\ref{12}) provides an exact representation, it contains an infinite number of coefficients, emanating from infinite numbers of orthonormal polynomials. In practice, the number of coefficients must be finite, say, by retaining at most $m$th-order polynomials in each variable. Furthermore, in many applications, the function
$y$ can be approximated by a sum of at most $S$-variate component functions, where $1\le S\le N$ is another truncation parameter, resulting in the $S$-variate, $m$th-order PDD approximation
\begin{equation}
\tilde{y}_{S,m}(\mathbf{X}) = y_{\emptyset}+{\displaystyle \sum_{{\textstyle {\emptyset\ne u\subseteq\{1,\cdots,N\}\atop 1\le|u|\le S}}}}\:\sum_{{\textstyle {\mathbf{j}_{|u|}\in\mathbb{N}_{0}^{|u|},\left\Vert \mathbf{j}_{|u|}\right\Vert _{\infty}\le m\atop j_{1},\cdots,j_{|u|}\neq0}}}C_{u\mathbf{j}_{|u|}}\psi_{u\mathbf{j}_{|u|}}(\mathbf{X}_{u}),
\label{14}
\end{equation}
which includes interactive effects of at most $S$ input variables $X_{i_{1}},\cdots,X_{i_{S}}$, $1\le i_{1}\le\cdots\le i_{S}\le N$, on $y$. Here, $\Vert \mathbf{j}_{|u|}\Vert _{\infty}:=\max (j_{i_1},\cdots,j_{i_{|u|}})$ represents the $\infty$-norm of $\mathbf{j}_{|u|}$. For instance, by selecting $S=1$ and $S=2$, the functions $\tilde{y}_{1,m}(\mathbf{X})$ and $\tilde{y}_{2,m}(\mathbf{X})$,
respectively, provide univariate and bivariate $m$th-order approximations,
contain contributions from all input variables, and should not be
viewed as first- and second-order approximations, nor do they limit
the nonlinearity of $y(\mathbf{X})$. Depending on how the
component functions are constructed, arbitrarily high-order univariate
and bivariate terms of $y(\mathbf{X})$ could be lurking
inside $\tilde{y}_{1,m}(\mathbf{X})$ and $\tilde{y}_{2,m}(\mathbf{X})$.
The fundamental conjecture underlying this decomposition is that the
component functions arising in the function decomposition will exhibit
insignificant $S$-variate interactions when $S\to N$, leading
to useful lower-variate approximations of $y(\mathbf{X})$.
When $S\to N$ and $m\to\infty$, $\tilde{y}_{S,m}(\mathbf{X})$
converges to $y(\mathbf{X})$ in the mean-square sense, that is, (\ref{14}) generates a hierarchical and convergent sequence of approximations of $y(\mathbf{X})$.  Further details of PDD are available elsewhere \cite{rahman09,rahman08}.

\subsubsection{Estimator}
Let $\tilde{Y}_{S,m}:=\tilde{y}_{S,m}(\mathbf{X})$ define the $S$-variate, $m$th-order PDD approximation of the output random variable $Y$.  Given the PDD truncation parameters $S$, $m$, and a sample size $L \in \mathbb{N}$, let $\{\mathbf{x}^{(l)},\tilde{\xi}_{S,m}^{(l)}\}_{l=1,\cdots,L}$ be the set of input-output sample pairs of $(\mathbf{X},\tilde{Y}_{S,m})$, where the output samples are calculated from (\ref{14}).  Using this surrogate sample set, the KDEs of the probability densities of $\tilde{Y}_{S,m}$ and $(\mathbf{X}_u,\tilde{Y}_{S,m})$ are obtained as
\begin{subequations}
\begin{equation}
\bar{f}_{\tilde{Y}_{S,m}}(\xi)=\dfrac{1}{L h_Y}\sum_{l=1}^L
K_Y \left( \dfrac{\xi-\tilde{\xi}_{S,m}^{(l)}}{h_Y} \right)
\label{15a}
\end{equation}
and
\begin{equation}
\bar{f}_{\mathbf{X}_u,\tilde{Y}_{S,m}}(\mathbf{x}_u,\xi)=
\dfrac{1}{L h_Y {\displaystyle \prod_{i \in u} h_{X_i}}}\sum_{l=1}^L
K_Y \left( \dfrac{\xi-\tilde{\xi}_{S,m}^{(l)}}{h_Y} \right)
{\displaystyle \prod_{i \in u}
K_{X_i} \left( \dfrac{x_i-x_i^{(l)}}{h_{X_i}} \right)
},
\label{15b}
\end{equation}
\end{subequations}
respectively, which are similar to Equations (\ref{10a}) and (\ref{10b}), but use instead the PDD approximation $\tilde{y}_{S,m}$ of $y$. Substituting the probability densities in Equation (\ref{9}) with their KDEs in (\ref{15a}) and (\ref{15b}) results in the PDD-KDE-MC estimator
\begin{equation}
\tilde{H}_{u,f}^{(L,S,m)} := \dfrac{1}{L}\sum_{l=1}^L
f\left(
\dfrac{\bar{f}_{Y}({\tilde{\xi}_{S,m}}^{(l)}){f}_{\mathbf{X}_u}(\mathbf{x}_u^{(l)})}
{\bar{f}_{\mathbf{X}_u,Y}(\mathbf{x}_u^{(l)},{\tilde{\xi}_{S,m}}^{(l)})}
\right)
\label{16}
\end{equation}
of $H_{u,f}$. From the well-known mean-square convergence properties \cite{rahman09,rahman08}, the sequence of PDD approximations $\tilde{y}_{S,m}$ also converges to $y$ in probability and in distribution as $S \to N$ and $m \to \infty$.  Therefore, the estimator $\tilde{H}_{u,f}^{(L,S,m)}$ in (\ref{16}) is anticipated to deliver a good approximation of ${H}_{u,f}$ when $L$ is sufficiently large, provided that the PDD truncation parameters $S$ and $m$ are chosen wisely.

The output samples from the PDD approximation should not be confused with those generated from the original function. The MC and KDE-MC methods, which require numerical calculations of $y$ for input samples, can be computationally expensive or even prohibitive, particularly when the sample size needs to be very large. In contrast, the samples generated from the PDD approximation $\tilde{y}_{S,m}$ entail inexpensive evaluations of simple polynomial functions. Therefore, a relatively large sample size can be accommodated in the PDD-KDE-MC method even when $y$ is expensive to evaluate.

It is important to emphasize that the PDD in (\ref{12}) and its truncation in (\ref{14}) are also valid for dependent random input, provided that the probability measure of $\mathbf{X}$ satisfies a few mild regulatory conditions \cite{rahman14}.  However, for the dependent input, not described here in detail for brevity, the evaluation of the expansion coefficients $C_{u\mathbf{j}_{|u|}}$ requires solving a system of linear matrix equations, where the construction of the coefficient matrix and coefficient vector entails a few $N$-dimensional integrals that are similar to the one in (\ref{13b}). See the author's recent work for further details \cite{rahman14}.

\subsection{ A few remarks}
It is important to note that the MC, KDE-MC, and PDD-KDE-MC methods proposed constitute single-loop (single simulation) samplings as opposed to double-loop computations required by several existing methods \cite{borgonovo07,liu09,wei13}. Therefore, the proposed estimators should be markedly more efficient than the double-loop methods even when computing existing sensitivity indices.

In all three methods presented, the input probability density $f_{\mathbf{X}}(\mathbf{x})$ of $\mathbf{X}$ is assumed to be known, so that the marginal density $f_{\mathbf{X}_u}(\mathbf{x}_u)$ of $\mathbf{X}_u$ for any $\emptyset\neq u\subseteq\{1,\cdots,N\}$ is exactly determined.  This assumption is commonly invoked or fulfilled in stochastic modeling and simulation where the objectives are propagating input uncertainties and hence determining the probabilistic characteristics of an output response.  However, there are also data-driven stochastic problems, where raw input data, generated from either physical testing or field measurements, are supplied.  In which case, the input density must also be estimated, say, by employing KDE to generate the approximate density $\bar{f}_{\mathbf{X}_u}(\mathbf{x}_u)$, therefore, imparting an added layer of approximation to all three methods.

From the central limit theorem, there exists a well-known probabilistic error bound of crude Monte Carlo simulation, revealing a slow convergence rate of $O(L^{-1/2})$, but independent of the dimension of the integral.  Therefore, the variance reduction techniques, such as importance sampling, stratified sampling, correlated sampling, and others, can be used to improve the efficiency of crude Monte Carlo simulation in all three methods.  In addition, if the convex function $f$ confers adequate smoothness to the integrand, then quasi-Monte Carlo sampling can be used instead of crude Monte Carlo sampling, accelerating the convergence rate, at most, to $O(L^{-1})$.

The KDE of probability densities required in the last two methods is known to suffer from the curse of dimensionality. Therefore, the methods presented here are limited to low-variate sensitivity indices $H_{u,f}$, that is, when $|u|$ is not overly large.  For an arbitrarily large $|u|$, the KDE becomes questionable, and alternative means of estimating probability densities, such as sparse-grid approximation \cite{peherstorfer12}, should be explored.  In addition, since only the ratio of probability densities is involved, more robust estimates of sensitivity indices are possible by invoking maximum likelihood estimation and others \cite{sugiyama12}.

\section{Numerical examples}
\subsection{Example 1}
The first example involves a linear transformation
\begin{equation}
y(\mathbf{X}) = X_1 + 1.1 X_2 + 1.2 X_3 + 1.3 X_4 + 1.4 X_5 + 1.5 X_6
\label{17}
\end{equation}
of six standard Gaussian random variables $X_i$, $i=1,\cdots,6$, which are independent and identically distributed with \emph{zero} means and \emph{unit} variances.  The larger the coefficient of a variable, the higher the importance of that variable.  Therefore, $X_1$ and $X_6$ are the least important and most important variables, respectively, in this problem.

Since the function $y$ in (\ref{17}) is a linear combination of Gaussian random variables, finding the exact probability density functions required to calculate $f$-sensitivity indices is elementary.  Figures \ref{figure1}(a) and \ref{figure1}(d) present two such exact density functions, the marginal density $f_Y(\xi)$ of $Y$ and the joint density $f_{X_1,Y}(x_1,\xi)$ of $(X_1,Y)$, respectively.  The joint densities involving five other input variables are similar and are, therefore, not shown or discussed. Figure \ref{figure1}(a) also depicts a comparison of the exact marginal density $f_Y(\xi)$ with the approximate marginal densities $\bar{f}_Y(\xi)$ obtained from KDE for two sample sizes: $L=10^4$ and $L=10^8$. The approximate joint densities $\bar{f}_{X_1,Y}(x_1,\xi)$, also generated from KDE using these two aforementioned sample sizes, are exhibited in Figures \ref{figure1}(b) and \ref{figure1}(c).  When the sample size increases, all KDE-generated densities, whether marginal or joint, approach the respective exact densities as expected.  In fact, the approximate and exact densities for $L=10^8$ are virtually indiscernible to the naked eye.

\begin{figure}
\begin{centering}
\includegraphics[scale=0.8]{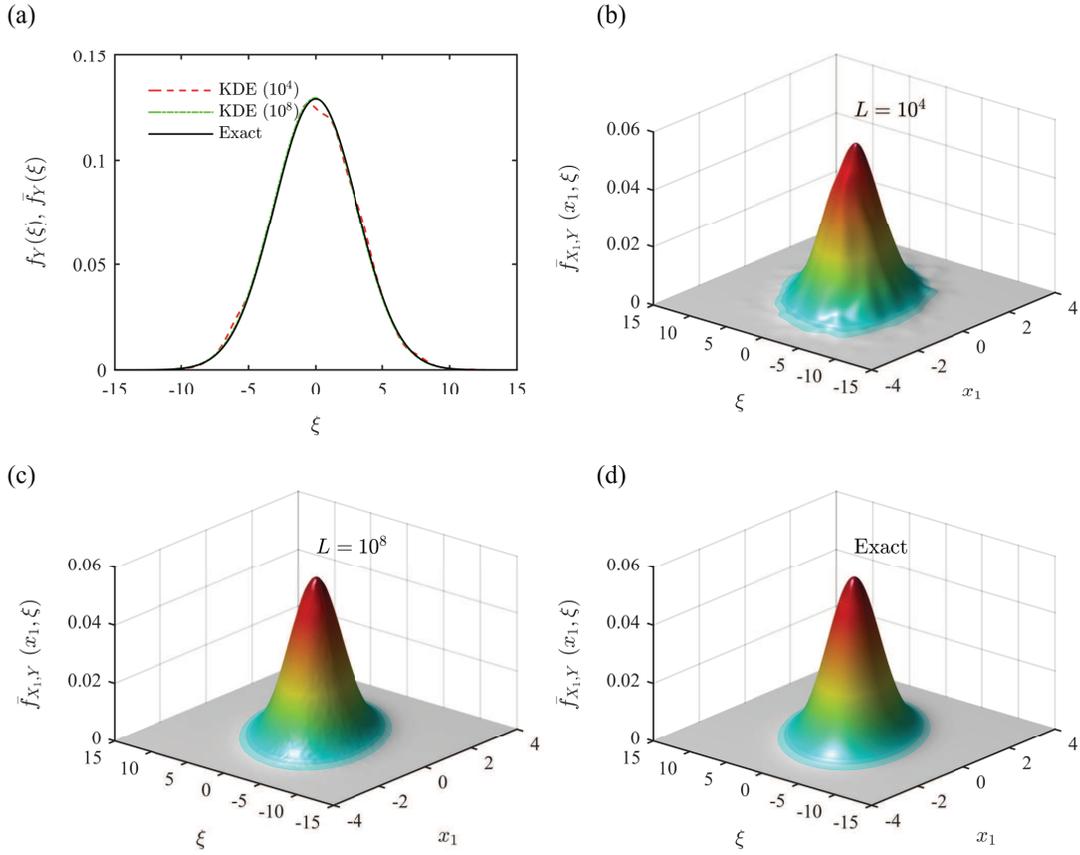}
\par\end{centering}
\caption{
Various probability density functions for independent random input in Example 1; (a) exact and KDE-generated marginal densities of $Y$; (b) KDE-generated joint density of $(X_1,Y)$ for $L=10^4$; (c) KDE-generated joint density of $(X_1,Y)$ for $L=10^8$; (d) exact joint density of $(X_1,Y)$.
}
\label{figure1}
\end{figure}

\begin{figure}
\begin{centering}
\includegraphics[scale=0.78]{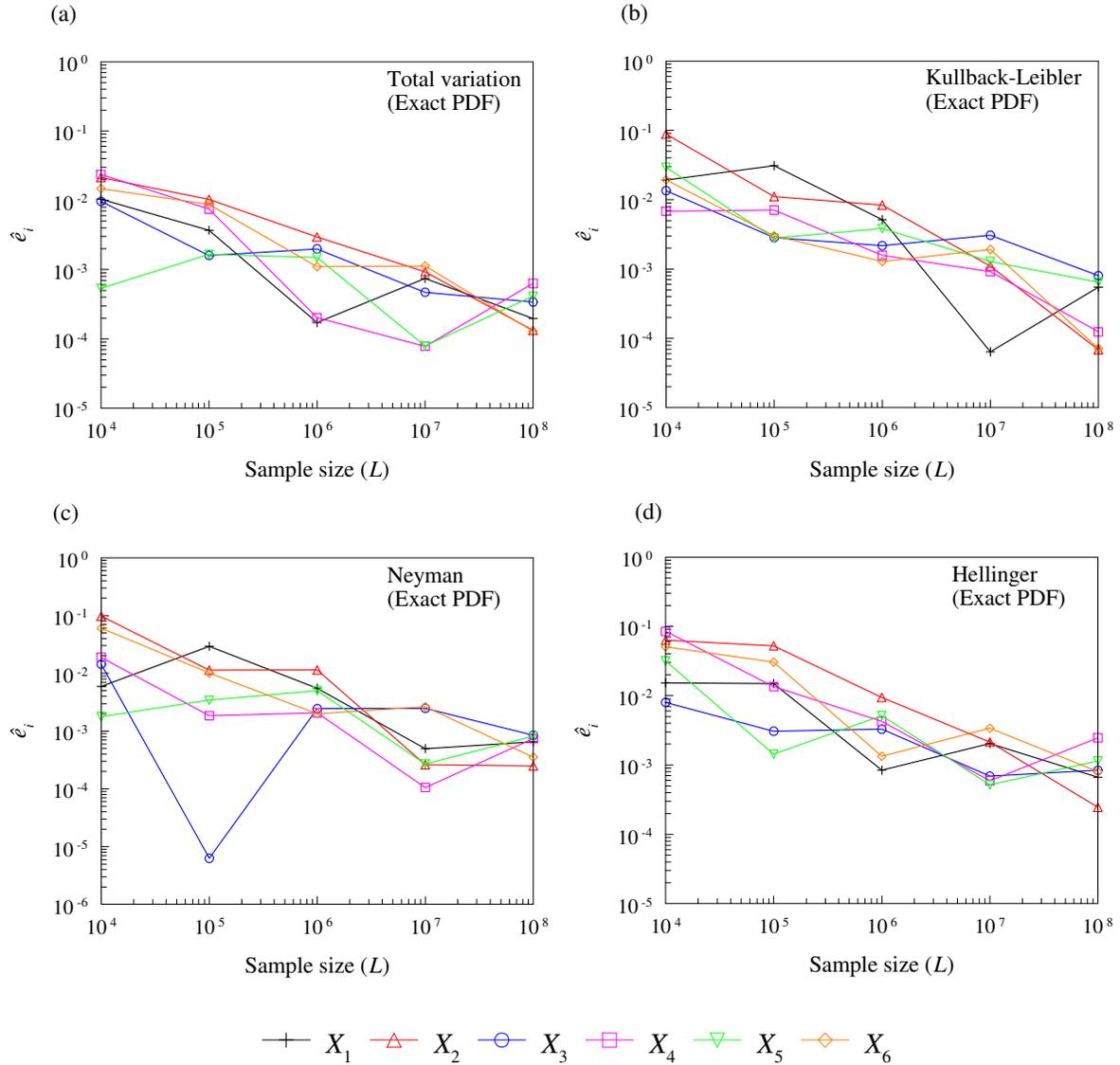}
\par\end{centering}
\caption{
${\cal L}_{1}$ errors committed by the MC method in estimating four variants of sensitivity indices for independent random input in Example 1; (a) $H_{\{i\},TV}$; (b) $H_{\{i\},KL'}$; (c) $H_{\{i\},N}$; (d) $H_{\{i\},H}$.}
\label{figure2}
\end{figure}

Figures \ref{figure2}(a) through \ref{figure2}(d) and Figures \ref{figure3}(a) through \ref{figure3}(d) display respectively the ${\cal L}_{1}$ errors $\hat{e}_i:=|H_{\{i\},f}-\hat{H}_{\{i\},f}^{(L)}|/H_{\{i\},f}$ by the MC method and $\bar{e}_i:=|H_{\{i\},f}-\bar{H}_{\{i\},f}^{(L)}|/H_{\{i\},f}$
by the KDE-MC method, in estimating the following four variants of univariate $f$-sensitivity indices for all six variables: (1) $H_{\{i\},TV}$ (total variational distance); (2) $H_{\{i\},KL'}$ (reversed Kullback-Leibler divergence); (3) $H_{\{i\},N}$ (Neyman $\chi^2$ divergence); and (4)  $H_{\{i\},H}$ (Hellinger distance).  However, calculating the sensitivity indices exactly is difficult, if not impossible.  Therefore, the reference solutions of $H_{\{i\},f}$ in calculating the errors were obtained using exact probability densities $f_Y(\xi)$ and $f_{X_i,Y}(x_i,\xi)$ in (\ref{8a}) through (\ref{8d}) and subsequent numerical integration.  The errors emanating from both methods drop with the sample size regardless of the variant of sensitivity indices examined.  However, for a given sample size, the errors committed by the MC method in general are lower than those perpetrated by the KDE-MC method.  This is due to kernel density estimations in the latter method.

\begin{figure}
\begin{centering}
\includegraphics[scale=0.78]{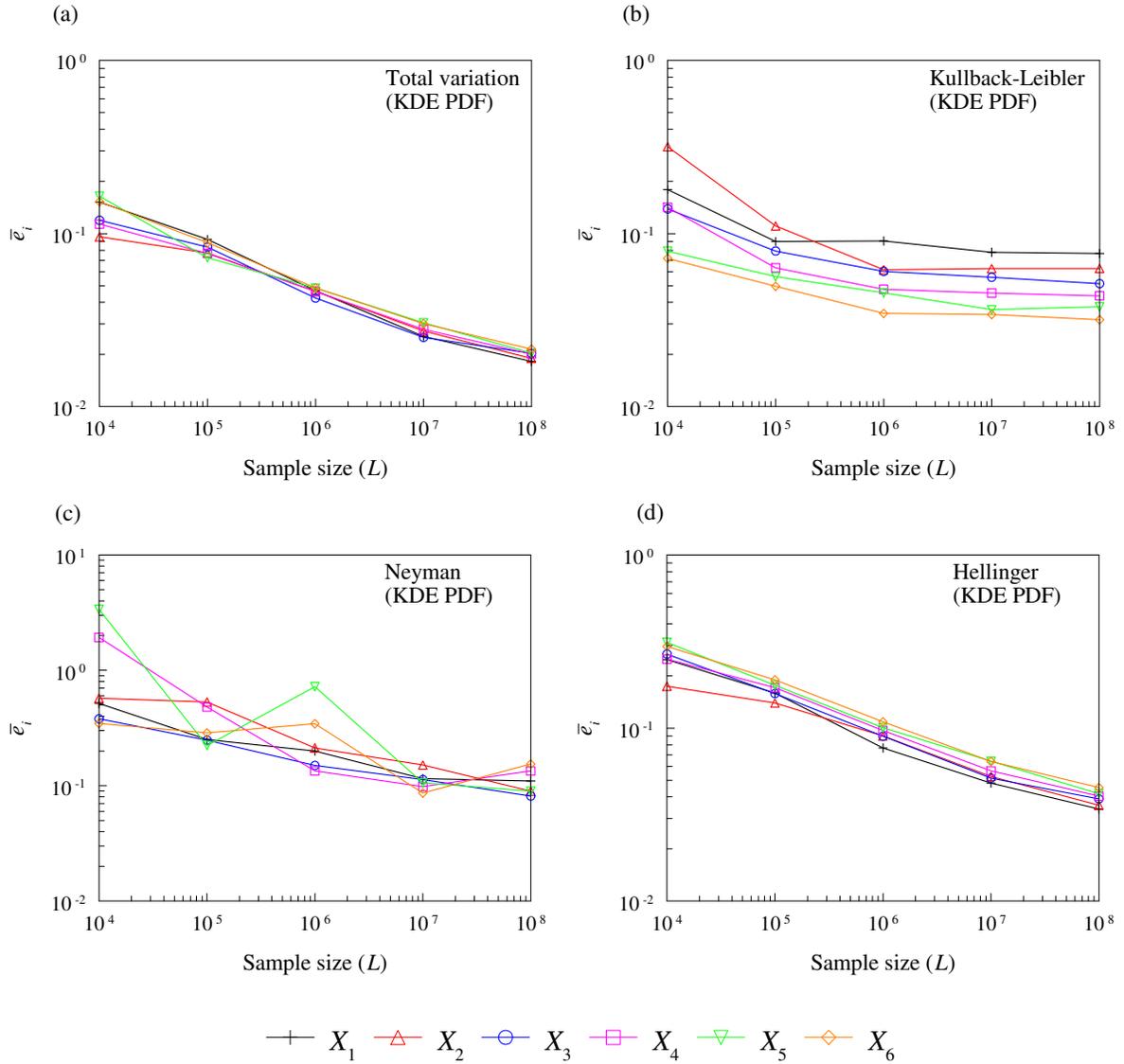}
\par\end{centering}
\caption{
${\cal L}_{1}$ errors committed by the KDE-MC method in estimating four variants of sensitivity indices for independent random input in Example 1; (a) $H_{\{i\},TV}$; (b) $H_{\{i\},KL'}$; (c) $H_{\{i\},N}$; (d) $H_{\{i\},H}$.
}
\label{figure3}
\end{figure}

Although the sensitivity analysis performed so far has adopted independence of input variables, the definition and calculation of $f$-sensitivity indices do not preclude dependent input variables. To corroborate this claim, consider the same function $y$ in (\ref{17}), but, this time, $y$ is subjected to a correlated Gaussian random input vector $\mathbf{X}=(X_1,\cdots,X_6) \in \mathbb{R}^6$ with a \emph{zero} mean and a covariance matrix $\mathbb{E}[X_i X_j]=\exp(-c|i-j|)$, $i,j=1,\cdots,6$, where the correlation parameter $c=1$. The MC and KDE-MC methods with a sample size of $L=10^8$ were employed to perform sensitivity analysis for this dependent input. The estimates of the same four variants of univariate sensitivity indices for all six variables are presented in Table \ref{table2}.  When compared with the results of numerical integration, also possible to obtain for dependent variables, both methods yield excellent sensitivity estimates in almost all cases. However, there are a few exceptions, such as those when the indices rooted in the Neyman $\chi^2$ divergence for the fifth and sixth variables are overpredicted by the KDE-MC method more than others. Again, this is due to the approximation error in the density estimation by the KDE.  Nonetheless, all four sensitivity analyses, regardless of the method, reach the same conclusion: the fourth and fifth variables are most important, whereas the first variable is least important.  Among the remaining three variables with intermediary significance, the third variable is relatively more important than either the second or the sixth variable. This is in contrast with the monotonically increasing significance of independent input variables. Indeed, the statistical dependence chosen --- in this case, an exponentially decaying correlation --- has altered the order of importance of input random variables.

\begin{table}
\caption{Estimates of sensitivity indices by the MC and KDE-MC methods for dependent random input in Example 1}
\begin{centering}
\begin{tabular}{cccccccc}
\hline
 & \multicolumn{3}{c}{{\footnotesize{}$\begin{array}{c}
{}\hat{H}_{\{i\},TV}^{(L)},\bar{H}_{\{i\},TV}^{(L)} \\
(\mathrm{total\: variational\: distance})
\end{array}$}} &  & \multicolumn{3}{c}{{\footnotesize{}$\begin{array}{c}
{}\hat{H}_{\{i\},KL'}^{(L)},\bar{H}_{\{i\},KL'}^{(L)} \\
(\mathrm{Kullback\: Leibler\:\mathrm{divergence}})
\end{array}$}}\tabularnewline
\cline{2-4} \cline{6-8}
{\footnotesize{}$\begin{array}{c}
\mathrm{Random}\\
\mathrm{variable}
\end{array}$} & {\footnotesize{}$\begin{array}{c}
\mathrm{MC}\\
\mathrm{method}
\end{array}$} & {\footnotesize{}$\begin{array}{c}
\mathrm{KDE-MC}\\
\mathrm{method}
\end{array}$} & {\footnotesize{}$\begin{array}{c}
\mathrm{Num.}\\
\mathrm{integ.}
\end{array}$} &  & {\footnotesize{}$\begin{array}{c}
\mathrm{MC}\\
\mathrm{method}
\end{array}$} & {\footnotesize{}$\begin{array}{c}
\mathrm{KDE-MC}\\
\mathrm{method}
\end{array}$} & {\footnotesize{}$\begin{array}{c}
\mathrm{Num.}\\
\mathrm{integ.}
\end{array}$}\tabularnewline
\hline
{\footnotesize{}$X_{1}$} & {\footnotesize{}0.278} & {\footnotesize{}0.272} & {\footnotesize{}0.278} &  & {\footnotesize{}0.086} & {\footnotesize{}0.090} & {\footnotesize{}0.086}\tabularnewline
{\footnotesize{}$X_{2}$} & {\footnotesize{}0.388} & {\footnotesize{}0.379} & {\footnotesize{}0.388} &  & {\footnotesize{}0.157} & {\footnotesize{}0.162} & {\footnotesize{}0.157}\tabularnewline
{\footnotesize{}$X_{3}$} & {\footnotesize{}0.463} & {\footnotesize{}0.451} & {\footnotesize{}0.463} &  & {\footnotesize{}0.215} & {\footnotesize{}0.219} & {\footnotesize{}0.215}\tabularnewline
{\footnotesize{}$X_{4}$} & {\footnotesize{}0.514} & {\footnotesize{}0.497} & {\footnotesize{}0.512} &  & {\footnotesize{}0.256} & {\footnotesize{}0.260} & {\footnotesize{}0.256}\tabularnewline
{\footnotesize{}$X_{5}$} & {\footnotesize{}0.513} & {\footnotesize{}0.499} & {\footnotesize{}0.514} &  & {\footnotesize{}0.258} & {\footnotesize{}0.262} & {\footnotesize{}0.258}\tabularnewline
{\footnotesize{}$X_{6}$} & {\footnotesize{}0.411} & {\footnotesize{}0.401} & {\footnotesize{}0.411} &  & {\footnotesize{}0.174} & {\footnotesize{}0.178} & {\footnotesize{}0.174}\tabularnewline
\hdashline
 & \multicolumn{3}{c}{{\footnotesize{}$\begin{array}{c}
{}\hat{H}_{\{i\},N}^{(L)},\bar{H}_{\{i\},N}^{(L)} \\
(\mathrm{Neyman\:\chi^{2}\: divergence})
\end{array}$}} &  & \multicolumn{3}{c}{{\footnotesize{}$\begin{array}{c}
{}\hat{H}_{\{i\},H}^{(L)},\bar{H}_{\{i\},H}^{(L)} \\
(\mathrm{Hellinger\: distance})
\end{array}$}}\tabularnewline
\cline{2-4} \cline{6-8}
 & {\footnotesize{}$\begin{array}{c}
\mathrm{MC}\\
\mathrm{method}
\end{array}$} & {\footnotesize{}$\begin{array}{c}
\mathrm{KDE-MC}\\
\mathrm{method}
\end{array}$} & {\footnotesize{}$\begin{array}{c}
\mathrm{Num.}\\
\mathrm{integ.}
\end{array}$} &  & {\footnotesize{}$\begin{array}{c}
\mathrm{MC}\\
\mathrm{method}
\end{array}$} & {\footnotesize{}$\begin{array}{c}
\mathrm{KDE-MC}\\
\mathrm{method}
\end{array}$} & {\footnotesize{}$\begin{array}{c}
\mathrm{Num.}\\
\mathrm{integ.}
\end{array}$}\tabularnewline
\hline
{\footnotesize{}$X_{1}$} & {\footnotesize{}0.187} & {\footnotesize{}0.205} & {\footnotesize{}0.187} &  & {\footnotesize{}0.045} & {\footnotesize{}0.043} & {\footnotesize{}0.045}\tabularnewline
{\footnotesize{}$X_{2}$} & {\footnotesize{}0.369} & {\footnotesize{}0.399} & {\footnotesize{}0.370} &  & {\footnotesize{}0.086} & {\footnotesize{}0.081} & {\footnotesize{}0.086}\tabularnewline
{\footnotesize{}$X_{3}$} & {\footnotesize{}0.536} & {\footnotesize{}0.575} & {\footnotesize{}0.536} &  & {\footnotesize{}0.120} & {\footnotesize{}0.113} & {\footnotesize{}0.120}\tabularnewline
{\footnotesize{}$X_{4}$} & {\footnotesize{}0.665} & {\footnotesize{}0.724} & {\footnotesize{}0.668} &  & {\footnotesize{}0.147} & {\footnotesize{}0.135} & {\footnotesize{}0.145}\tabularnewline
{\footnotesize{}$X_{5}$} & {\footnotesize{}0.675} & {\footnotesize{}0.876} & {\footnotesize{}0.674} &  & {\footnotesize{}0.145} & {\footnotesize{}0.137} & {\footnotesize{}0.146}\tabularnewline
{\footnotesize{}$X_{6}$} & {\footnotesize{}0.416} & {\footnotesize{}0.575} & {\footnotesize{}0.416} &  & {\footnotesize{}0.096} & {\footnotesize{}0.091} & {\footnotesize{}0.095}\tabularnewline
\hline
\end{tabular}
\par\end{centering}
\label{table2}
\end{table}

\subsection{Example 2}
The next example originates from a probabilistic risk assessment model, addressed by Iman \cite{iman87}, to determine the importance of input random variables.  The model response is expressed by
\begin{equation}
\begin{array}{rcl}
y(\mathbf{X}) & = & X_1 X_3 X_5 + X_1 X_3 X_6 + X_1 X_4 X_5 + X_1 X_4 X_6 + X_2 X_3 X_4  \\
              &   & + X_2 X_3 X_5 + X_2 X_4 X_5 + X_2 X_5 X_6 + X_2 X_4 X_7 + X_2 X_6 X_7,
\end{array}
\label{18}
\end{equation}
where $X_i$, $i=1,\cdots,7$, are seven independent and identically distributed lognormal random variables with their statistical properties described in Table \ref{table3}. Borgonovo \cite{borgonovo07}, Liu and Homma \cite{liu09}, and Wei, Lu, and Yuan \cite{wei13} analyzed the function $y$ in (\ref{18}), calculating the total variational sensitivity indices from various approximate techniques.

\begin{table}
\caption{Statistical properties of input variables in Example 2}
\begin{centering}
\begin{tabular}{cccc}
\hline
{\footnotesize{Random variable}} & {\footnotesize{Probability distribution}} & {\footnotesize{Mean}} & {\footnotesize{Error factor}}\tabularnewline
\hline
{\footnotesize{$X_{1}$}} & {\footnotesize{Lognormal}} & {\footnotesize{2}} & {\footnotesize{2}}\tabularnewline
{\footnotesize{$X_{2}$}} & {\footnotesize{Lognormal}} & {\footnotesize{3}} & {\footnotesize{2}}\tabularnewline
{\footnotesize{$X_{3}$}} & {\footnotesize{Lognormal}} & {\footnotesize{0.001}} & {\footnotesize{2}}\tabularnewline
{\footnotesize{$X_{4}$}} & {\footnotesize{Lognormal}} & {\footnotesize{0.002}} & {\footnotesize{2}}\tabularnewline
{\footnotesize{$X_{5}$}} & {\footnotesize{Lognormal}} & {\footnotesize{0.004}} & {\footnotesize{2}}\tabularnewline
{\footnotesize{$X_{6}$}} & {\footnotesize{Lognormal}} & {\footnotesize{0.005}} & {\footnotesize{2}}\tabularnewline
{\footnotesize{$X_{7}$}} & {\footnotesize{Lognormal}} & {\footnotesize{0.003}} & {\footnotesize{2}}\tabularnewline
\hline
\end{tabular}
\par\end{centering}
\label{table3}
\end{table}

\begin{table}
\caption{Various estimates of scaled sensitivity indices from the total variational distance in Example 2}
\begin{centering}
{\footnotesize{}}%
\begin{tabular}{cccccc}
\hline
{\footnotesize{$\begin{array}{c}
\mathrm{}\\
\mathrm{Random}\\
\mathrm{variable}
\end{array}$}} & \!\!\!\!\!{\footnotesize{$\begin{array}{c}
\mathrm{KDE-MC}\\
\mathrm{method}\\
(L=10^{4})
\end{array}$}} & \!\!\!\!\!{\footnotesize{$\begin{array}{c}
\mathrm{KDE-MC}\\
\mathrm{method}\\
(L=10^{6})
\end{array}$}} & \!\!\!\!\!{\footnotesize{$\begin{array}{c}
\mathrm{\mathrm{Borgonovo's}}\\
\mathrm{method}\\{}
\cite{borgonovo07}
\end{array}$}} & \!\!\!\!\!{\footnotesize{$\begin{array}{c}
\mathrm{\mathrm{Liu\:\&\: Homma's}}\\
\mathrm{PDF-based}\\
\mathrm{method}\:\cite{liu09}
\end{array}$}} & \!\!\!\!\!{\footnotesize{$\begin{array}{c}
\mathrm{\mathrm{Liu\:\&\: Homma's}}\\
\mathrm{CDF-based}\\
\mathrm{method}\:\cite{liu09}
\end{array}$}}\tabularnewline
\hline
{\footnotesize{$X_{1}$}} & \!\!\!\!\!{\footnotesize{0.07 (6)}} & \!\!\!\!\!{\footnotesize{0.07 (6)}} & \!\!\!\!\!{\footnotesize{0.11 (6)}} & \!\!\!\!\!{\footnotesize{0.07 (6)}} & \!\!\!\!\!{\footnotesize{0.10 (6)}}\tabularnewline
{\footnotesize{$X_{2}$}} & \!\!\!\!\!{\footnotesize{0.19 (1)}} & \!\!\!\!\!{\footnotesize{0.21 (1)}} & \!\!\!\!\!{\footnotesize{0.17 (3)}} & \!\!\!\!\!{\footnotesize{0.23 (1)}} & \!\!\!\!\!{\footnotesize{0.24 (1)}}\tabularnewline
{\footnotesize{$X_{3}$}} & \!\!\!\!\!{\footnotesize{0.06 (7)}} & \!\!\!\!\!{\footnotesize{0.04 (7)}} & \!\!\!\!\!{\footnotesize{0.09 (7)}} & \!\!\!\!\!{\footnotesize{0.05 (7)}} & \!\!\!\!\!{\footnotesize{0.08 (7)}}\tabularnewline
{\footnotesize{$X_{4}$}} & \!\!\!\!\!{\footnotesize{0.10 (4)}} & \!\!\!\!\!{\footnotesize{0.10 (4)}} & \!\!\!\!\!{\footnotesize{0.13 (4)}} & \!\!\!\!\!{\footnotesize{0.11 (4)}} & \!\!\!\!\!{\footnotesize{0.13 (4)}}\tabularnewline
{\footnotesize{$X_{5}$}} & \!\!\!\!\!{\footnotesize{0.13 (3)}} & \!\!\!\!\!{\footnotesize{0.14 (3)}} & \!\!\!\!\!{\footnotesize{0.18 (2)}} & \!\!\!\!\!{\footnotesize{0.15 (3)}} & \!\!\!\!\!{\footnotesize{0.16 (3)}}\tabularnewline
{\footnotesize{$X_{6}$}} & \!\!\!\!\!{\footnotesize{0.15 (2)}} & \!\!\!\!\!{\footnotesize{0.16 (2)}} & \!\!\!\!\!{\footnotesize{0.20 (1)}} & \!\!\!\!\!{\footnotesize{0.18 (2)}} & \!\!\!\!\!{\footnotesize{0.19 (2)}}\tabularnewline
{\footnotesize{$X_{7}$}} & \!\!\!\!\!{\footnotesize{0.07 (5)}} & \!\!\!\!\!{\footnotesize{0.07 (5)}} & \!\!\!\!\!{\footnotesize{0.11 (5)}} & \!\!\!\!\!{\footnotesize{0.08 (5)}} & \!\!\!\!\!{\footnotesize{0.10 (5)}}\tabularnewline
\hline
\end{tabular}
\par\end{centering}
\label{table4}
\end{table}

Table \ref{table4} presents estimates of scaled univariate sensitivity indices, $\bar{H}_{\{i\},TV}^{(L)}/2$, $i=1,\cdots,7$, by the KDE-MC method for small and large sample sizes: $L=10^4$ and $L=10^6$. The indices, derived from the total variational distance, are divided by \emph{two} as per Borgonovo's importance measure.  Although the sensitivity estimates for most variables vary slightly with respect to the sample size, the relative rankings are identical, which are marked inside the parentheses in Table \ref{table4}.  For the sake of comparison, the estimates of sensitivity indices published by Borgonovo and Liu and Homma are also listed in Table \ref{table4}.  The estimates by Liu and Homma are further broken down according to the PDF- and CDF-based methods discussed in their paper.  It appears that the proposed KDE-MC method ($L=10^6$) and Liu and Homma's PDF-based method yield very close sensitivity estimates.  In contrast, Borgonovo's method and Liu and Homma's CDF-based method predict mostly larger sensitivity indices than those estimated by the KDE-MC method.  This is potentially due to a poor fit of classical distributions, employed by Borgonovo, to estimate the probability densities of the model response.  Nonetheless, the proposed KDE-MC method and both of Liu and Homma's methods lead to the same relative rankings, whereas a slight discrepancy exists in the rankings from Borgonovo's method. Although no numerical results of sensitivity indices are given, Wei, Lu, and Yuan reported the same relative rankings as derived from the proposed KDE-MC method.

\subsection{Example 3}
In the third example, consider the function
\begin{equation}
y(\mathbf{X})=\sin X_{1}+7\sin^{2}X_{2}+0.1X_{3}^{4}\sin X_{1},
\label{19}
\end{equation}
studied by Ishigami and Homma \cite{ishigami90}, where $X_i$, $i=1,2,3$, are three independent and identically distributed uniform random variables on $[-\pi,+\pi]$.  The function $y$ in (\ref{19}) with various coefficients has been widely used for variance-based global sensitivity analysis \cite{rahman11,ratto07,sobol93,tarantola06}.

\begin{figure}[!b]
\begin{centering}
\includegraphics[scale=0.8]{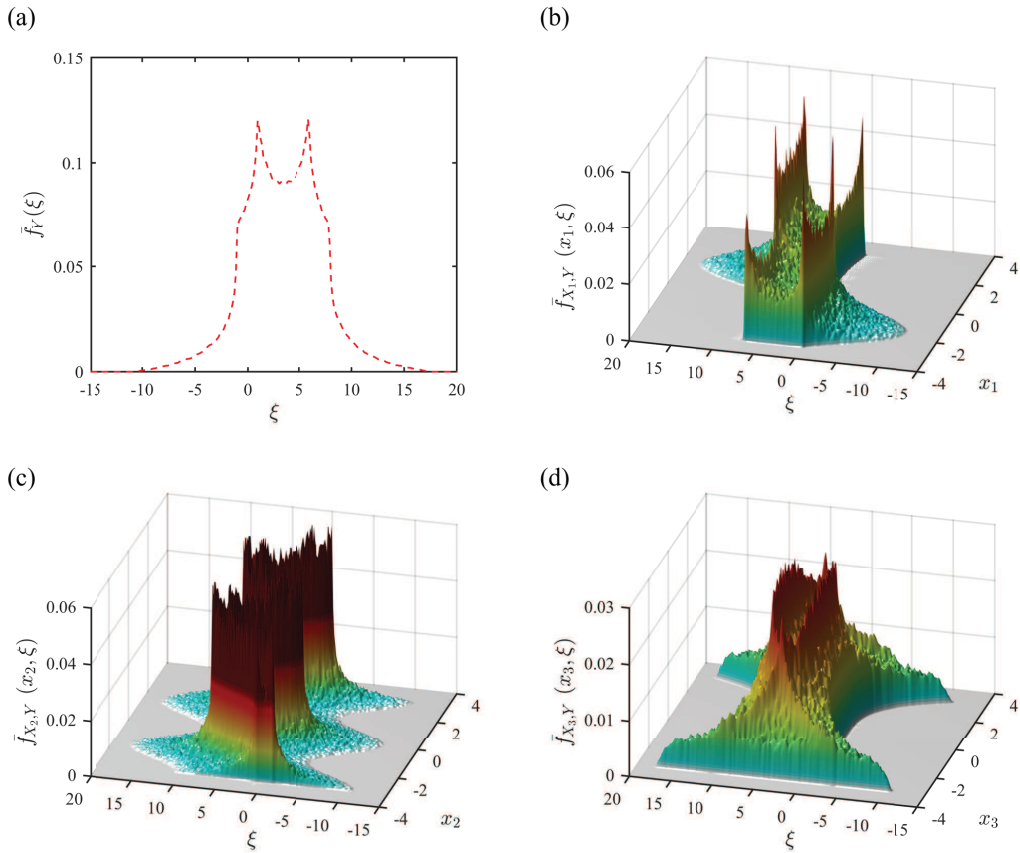}
\par\end{centering}
\caption{
KDE-generated probability density functions in Example 3 for $L=10^6$; (a) marginal density of $Y$; (b) joint density of $(X_1,Y)$; (c) joint density of $(X_2,Y)$; (d) joint density of $(X_3,Y)$.
}
\label{figure4}
\end{figure}

The KDE with three distinct sample sizes, $L=10^4, 10^5, 10^6$, was employed to generate the marginal probability density function $\bar{f}_Y(\xi)$ and three joint probability density functions  $\bar{f}_{X_1,Y}(x_1,\xi)$, $\bar{f}_{X_2,Y}(x_2,\xi)$, and $\bar{f}_{X_3,Y}(x_3,\xi)$. For brevity, the four density functions, obtained only for the largest sample size, that is, for $L=10^6$, are displayed in Figures \ref{figure4}(a) through \ref{figure4}(d). The density functions, especially the ones in Figures \ref{figure4}(b), \ref{figure4}(c) and \ref{figure4}(d), have complex features, demonstrating a need for large sample sizes to warrant adequate accuracy in their estimations.  The sample size, $L=10^6$, is sufficiently large, meaning that a larger sample size produces only negligible changes in the density functions.  Therefore, a global sensitivity analysis using $L=10^6$ should yield reliable estimates.

Table \ref{table5} presents estimates of two variants of density-based univariate sensitivity indices, $\bar{H}_{\{i\},TV}^{(L)}$ and $\bar{H}_{\{i\},KL'}^{(L)}$, which are derived from the total variational distance and reversed Kullback-Leibler divergence, respectively, and the variance-based Sobol indices.  The density-based sensitivity indices were estimated using the KDE-MC method for all three sample sizes, whereas the Sobol indices were obtained exactly.  The density-based sensitivity indices, whether $\bar{H}_{\{i\},TV}^{(L)}$ or $\bar{H}_{\{i\},KL'}^{(L)}$, level off as $L$ increases.  Regardless of the sample size, both variants of sensitivity analysis lead to the same conclusion: $X_2$ is the most important variable, followed by $X_3$ and $X_1$ or $X_1$ and $X_3$, depending on which type of distance or divergence is employed or appropriate. The relative rankings from Sobol univariate (first-order) indices, which are strictly variance-based, also find $X_2$ to be the most important variable.  However, the Sobol total indices tell a different tale: $X_1$ and $X_3$ are the most important and least important variables, respectively, whereas the significance of $X_2$ is intermediary.  This is because the interaction effect of $X_1$ and $X_3$ is accounted for in the Sobol total indices.  All other indices, whether density- or variance-based, capture only the main effects.

\begin{table}
\caption{KDE-MC estimates of sensitivity indices from the total variational distance and reversed Kullback-Leibler divergence and exact Sobol indices in Example 3}
\begin{centering}
{\footnotesize{}}%
\begin{tabular}{ccccccccccc}
\hline
 & \multicolumn{3}{c}{{\footnotesize{$\begin{array}{c}\!\!\!\!\!\!\!
{}\bar{H}_{\{i\},TV}^{(L)} \\
(\mathrm{total\: variational\: distance})
\end{array}$}}} &  & \multicolumn{3}{c}{{\footnotesize{$\begin{array}{c}\!\!\!\!\!\!\!
{}\bar{H}_{\{i\},KL'}^{(L)} \\
(\mathrm{Kullback-\: Leibler\:\mathrm{divergence}})
\end{array}$}}} &  & \multicolumn{2}{c}{{\footnotesize{$\begin{array}{c}\!\!\!\!\!\!\!
S_{\{i\}} \\
(\mathrm{Sobol\: index})
\end{array}$}}}\tabularnewline
\cline{2-4} \cline{6-8} \cline{10-11}
{\footnotesize{$\begin{array}{c}\!\!\!\!\!\!\!
\mathrm{Random}\\\!\!\!\!\!\!\!
\mathrm{variable}
\end{array}$}} & \!\!\!\!{\footnotesize{$L=10^{4}$}} & \!\!\!\!{\footnotesize{$L=10^{5}$}} & \!\!\!\!{\footnotesize{$L=10^{6}$}} &  & \!\!\!\!{\footnotesize{$L=10^{4}$}} & \!\!\!\!{\footnotesize{$L=10^{5}$}} & \!\!\!\!{\footnotesize{$L=10^{6}$}} &  & \!\!\!\!{\footnotesize{Univariate}} & \!\!\!\!{\footnotesize{Total}}\tabularnewline
\hline
\!\!\!\!\!\!\!{\footnotesize{$X_{1}$}} & \!\!\!\!{\footnotesize{0.299}} & \!\!\!\!{\footnotesize{0.298}} & \!\!\!\!{\footnotesize{0.288}} &  & \!\!\!\!{\footnotesize{0.358}} & \!\!\!\!{\footnotesize{0.369}} & \!\!\!\!{\footnotesize{0.366}} &  & \!\!\!\!{\footnotesize{0.314}} & \!\!\!\!{\footnotesize{0.558}}\tabularnewline
\!\!\!\!\!\!\!{\footnotesize{$X_{2}$}} & \!\!\!\!{\footnotesize{0.729}} & \!\!\!\!{\footnotesize{0.797}} & \!\!\!\!{\footnotesize{0.828}} &  & \!\!\!\!{\footnotesize{0.522}} & \!\!\!\!{\footnotesize{0.505}} & \!\!\!\!{\footnotesize{0.483}} &  & \!\!\!\!{\footnotesize{0.442}} & \!\!\!\!{\footnotesize{0.442}}\tabularnewline
\!\!\!\!\!\!\!{\footnotesize{$X_{3}$}} & \!\!\!\!{\footnotesize{0.291}} & \!\!\!\!{\footnotesize{0.303}} & \!\!\!\!{\footnotesize{0.306}} &  & \!\!\!\!{\footnotesize{0.193}} & \!\!\!\!{\footnotesize{0.219}} & \!\!\!\!{\footnotesize{0.232}} &  & \!\!\!\!{\footnotesize{0}} & \!\!\!\!{\footnotesize{0.244}}\tabularnewline
\hline
\end{tabular}
\par\end{centering}
\label{table5}
\end{table}

To determine the accuracy of the PDD-KDE-MC method, the density-based univariate sensitivity indices discussed in the preceding paragraph were estimated using several PDD approximations of $y$. Legendre orthonormal polynomials, which are consistent with uniform probability distributions of input random variables, were employed for the PDD approximation. Since the right-hand side of (\ref{19}) includes interactive effects of at most two variables, the bivariate PDD approximation is adequate.  However, $y$ is a non-polynomial function; therefore, several polynomial orders are required to study convergence.  Table \ref{table6} lists the estimates of the same two variants of density-based univariate sensitivity indices $\tilde{H}_{\{i\},TV}^{(L,S,m)}$ and $\tilde{H}_{\{i\},KL'}^{(L,S,m)}$, obtained by the PDD-KDE-MC method for $L=10^6$, $S=2$, and $m=4, 6, 8$.  The sensitivity indices for each variable converge with respect to $m$.  When $m=8$, the indices generated by the PDD-KDE-MC method in Table \ref{table6} are very close to those obtained by the KDE-MC method ($L=10^6$) in Table \ref{table5}.  The computational efforts by the PDD-KDE-MC method, measured in terms of the number of function evaluations and listed in Table \ref{table6}, vary from 61 to 217, depending on the values of $m$ chosen; nonetheless, they are markedly lower than the $10^6$ function evaluations required by the KDE-MC method.  Therefore, for sensitivity analysis of complex systems, where $y$ is expensive to evaluate, the KDE-MC method will not be practical and a method like the PDD-KDE-MC method becomes necessary; it will be demonstrated next.

\begin{table}
\caption{PDD-KDE-MC estimates of sensitivity indices from the total variational distance and reversed Kullback-Leibler divergence in Example 3}
\begin{centering}
{\footnotesize{}}%
\begin{tabular}{cccccccc}
\hline
 & \multicolumn{3}{c}{{\footnotesize{$\begin{array}{c}
{}\tilde{H}_{\{i\},TV}^{(L,S,m)},\, L=10^{6},\, S=2 \\
(\mathrm{total\: variational\: distance})
\end{array}$}}} &  & \multicolumn{3}{c}{{\footnotesize{$\begin{array}{c}
{}\tilde{H}_{\{i\},KL'}^{(L,S,m)},\, L=10^{6},\, S=2 \\
(\mathrm{Kullback-\: Leibler\:\mathrm{divergence}})
\end{array}$}}}\tabularnewline
\cline{2-4} \cline{6-8}
{\footnotesize{$\begin{array}{c}
\mathrm{Random}\\
\mathrm{variable}
\end{array}$}} & {\footnotesize{$m=4$}} & {\footnotesize{$m=6$}} & {\footnotesize{$m=8$}} &  & {\footnotesize{$m=4$}} & {\footnotesize{$m=6$}} & {\footnotesize{$m=8$}}\tabularnewline
\hline
{\footnotesize{$X_{1}$}} & {\footnotesize{0.336}} & {\footnotesize{0.287}} & {\footnotesize{0.287}} &  & {\footnotesize{0.264}} & {\footnotesize{0.346}} & {\footnotesize{0.343}}\tabularnewline
{\footnotesize{$X_{2}$}} & {\footnotesize{0.878}} & {\footnotesize{0.857}} & {\footnotesize{0.832}} &  & {\footnotesize{0.630}} & {\footnotesize{0.513}} & {\footnotesize{0.484}}\tabularnewline
{\footnotesize{$X_{3}$}} & {\footnotesize{0.288}} & {\footnotesize{0.300}} & {\footnotesize{0.303}} &  & {\footnotesize{0.150}} & {\footnotesize{0.218}} & {\footnotesize{0.221}}\tabularnewline
{\footnotesize{$\begin{array}{c}
\mathrm{No.\: of\: function}\\
\mathrm{evaluations^{(a)}}
\end{array}$}} & {\footnotesize{61}} & {\footnotesize{127}} & {\footnotesize{217}} &  & {\footnotesize{61}} & {\footnotesize{127}} & {\footnotesize{217}}\tabularnewline
\hline
\end{tabular}
\par\end{centering}{\footnotesize \par}
{\scriptsize{~~~~~~~~~~~~~(a) No. of function evaluations = $N(N-1)m^2/2+Nm+1$.}}
\label{table6}
\end{table}

\subsection{Example 4}
The final example illustrates the PDD-KDE-MC method for sensitivity analysis of an industrial-scale, stochastic mechanics problem. It involves linear-elastic stress analysis of a leverarm in a wheel loader, depicted in Figure \ref{figure5}(a), commonly used in the heavy construction industry. The material body of the leverarm occupies domain $\mathcal{D}\subset\mathbb{R}^{3}$ with boundary $\Gamma$. At a spatial point $\mathbf{s}\in\mathcal{D}$, it is subjected to random surface traction $\mathbf{\bar{t}}(\mathbf{s};\mathbf{X})$ on $\Gamma_{t}$ and random displacement $\mathbf{\bar{u}}(\mathbf{s};\mathbf{X})$ on $\Gamma_{u}$, where $\Gamma_{t}\cap\Gamma_{u}=\emptyset$, with $\Gamma_{t}$ and $\Gamma_{u}$ denoting the parts of the boundary where natural (Neumann) and essential (Dirichlet) boundary conditions are enforced. Then the strong form of the governing equations for a stochastic boundary-value problem in elastostatics calls for finding the displacement $\mathbf{u}(\mathbf{s};\mathbf{X})$
and stress $\boldsymbol{\sigma}(\mathbf{s};\mathbf{X})$ responses, which satisfy $P$-almost surely
\begin{equation}
\begin{array}{rcl}
\boldsymbol{\nabla}\cdot\boldsymbol{\sigma}(\mathbf{s};\mathbf{X})+\mathbf{b}(\mathbf{s};\mathbf{X}) & = & \mathbf{0}\;\mathrm{in}\;\mathcal{D},\\
\boldsymbol{\sigma}(\mathbf{s};\mathbf{X})\cdot\mathbf{n}(\mathbf{s};\mathbf{X}) & = & \mathbf{\bar{t}}(\mathbf{s};\mathbf{X})\;\mathrm{on}\;\Gamma_{t},\\
\mathbf{u}(\mathbf{s};\mathbf{X}) & = & \mathbf{\bar{u}}(\mathbf{s};\mathbf{X})\;\mathrm{on}\;\Gamma_{u},
\end{array}
\label{20}
\end{equation}
where $\boldsymbol{\nabla}:=\{\partial/\partial s_{1},\partial/\partial s_{2},\partial/\partial s_{3}\}$, $\mathbf{n}(\mathbf{s};\mathbf{X})$ is the unit outward normal vector, $\boldsymbol{\sigma}(\mathbf{s};\mathbf{X})=\mathbf{D}(\mathbf{X})\boldsymbol{\epsilon}(\mathbf{s};\mathbf{X})$
is the random stress vector with $\mathbf{D}(\mathbf{X})$ representing the random elasticity matrix, and $\boldsymbol{\epsilon}(\mathbf{s};\mathbf{X})=\boldsymbol{\nabla}_{s}\mathbf{u}(\mathbf{s};\mathbf{X})$
is the random strain vector with $\boldsymbol{\nabla}_{s}$ representing the symmetric part of $\boldsymbol{\nabla}\mathbf{u}(\mathbf{s};\mathbf{X})$.  The finite-element method was used to solve the variational weak form of (\ref{20}).

\begin{figure}[!h]
\begin{centering}
\includegraphics[scale=0.7]{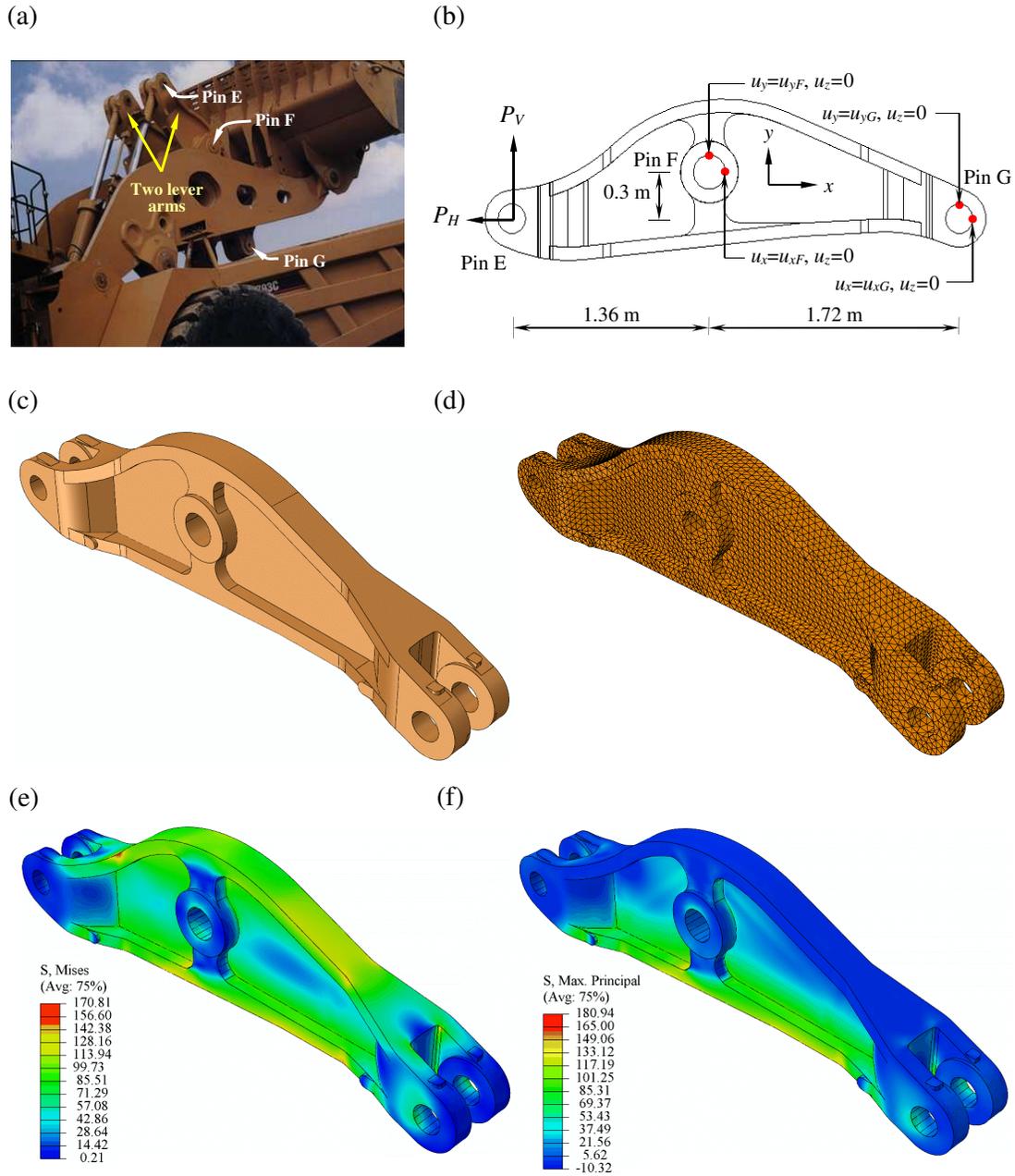}
\par\end{centering}
\caption{
Linear-elastic stress analysis of a leverarm; (a) two leverarms in a wheel loader; (b) geometry and boundary conditions; (c) computer-aided design model; (d) finite-element grid (48,312 elements); (e) von Mises stress contours for a sample input; (f) maximum principal stress contours for a sample input.
}
\label{figure5}
\end{figure}

The loading and boundary conditions of a single leverarm are shown in Figure \ref{figure5}(b). Figure \ref{figure5}(c) and \ref{figure5}(d) present a computer-aided design model and a finite-element grid comprising 48,312 second-order tetrahedral elements of the leverarm, respectively. Two random loads $P_{H}$ and $P_{V}$ acting at pin E can be viewed as input loads due to other mechanical components of the wheel loader. The essential boundary conditions, sketched in Figure \ref{figure5}(b), define random prescribed displacements $u_{xF}$ and $u_{yF}$ at pin F and $u_{xG}$, and $u_{yG}$ at pin G. The leverarm is made of cast steel with random Young's modulus $E$ and random Poisson's ratio $\nu$. The input vector $\boldsymbol{X}=(P_{H},P_{V},E,\nu,u_{xF},u_{yF},u_{xG},u_{yG})\in\mathbb{R}^{8}$ comprises eight independent random variables with their statistical properties specified in Table \ref{table7}. The von Mises stress and maximum principal stress distributions in Figure \ref{figure5}(e) and \ref{figure5}(f), respectively, calculated for an arbitrarily selected sample input, are commonly used for examining material yielding or fatigue damage in mechanical systems.  Both the univariate ($S=1$) and bivariate ($S=2$) PDD approximations with measure-consistent orthogonal polynomials and Gauss quadrature rule \cite{rahman09} were employed for two density-based sensitivity analyses of a stress response from FEA of the leverarm.

\begin{table}
\caption{Statistical properties of input variables in Example 4}
\begin{centering}
\begin{tabular}{cccc}
\hline
{\footnotesize{Random variable }} & {\footnotesize{Probability distribution}} & {\footnotesize{Mean}} & {\footnotesize{Standard deviation}}\tabularnewline
\hline
{\footnotesize{$P_{H}{}^{(\mathrm{a})}$, kN}} & {\footnotesize{Lognormal}} & {\footnotesize{507.69}} & {\footnotesize{76.15}}\tabularnewline
{\footnotesize{$P_{V}{}^{(\mathrm{a})}$, kN}} & {\footnotesize{Lognormal}} & {\footnotesize{1517.32}} & {\footnotesize{227.60}}\tabularnewline
{\footnotesize{$E$, GPa}} & {\footnotesize{Lognormal}} & {\footnotesize{203}} & {\footnotesize{10.15}}\tabularnewline
{\footnotesize{$\nu$}} & {\footnotesize{Lognormal}} & {\footnotesize{0.3}} & {\footnotesize{0.015}}\tabularnewline
{\footnotesize{$u_{xF}$, mm}} & {\footnotesize{Uniform$^{(\mathrm{b})}$}} & {\footnotesize{-5}} & {\footnotesize{$5/\sqrt{3}$}}\tabularnewline
{\footnotesize{$u_{yF}$, mm}} & {\footnotesize{Uniform$^{(\mathrm{c})}$}} & {\footnotesize{5}} & {\footnotesize{$5/\sqrt{3}$}}\tabularnewline
{\footnotesize{$u_{xG}$, mm}} & {\footnotesize{Uniform$^{(\mathrm{c})}$}} & {\footnotesize{5}} & {\footnotesize{$5/\sqrt{3}$}}\tabularnewline
{\footnotesize{$u_{yG}$, mm}} & {\footnotesize{Uniform$^{(\mathrm{b})}$}} & {\footnotesize{-5}} & {\footnotesize{$5/\sqrt{3}$}}\tabularnewline
\hline
\end{tabular}
\par\end{centering}{\footnotesize \par}

{\scriptsize{~~~~~~~~~~~~~~~~~~~~~(a) To be distributed equally (halved)
on the front and back sides of pin E.}}{\scriptsize \par}

{\scriptsize{~~~~~~~~~~~~~~~~~~~~~(b) Uniformly distributed over
$[-10,0]$ mm; to be applied on both sides.}}{\scriptsize \par}

{\scriptsize{~~~~~~~~~~~~~~~~~~~~~(c) Uniformly distributed over
$[0,10]$ mm; to be applied on both sides.}}
\label{table7}
\end{table}

\begin{table}
\caption{PDD-KDE-MC estimates of sensitivity indices from the total variational distance and reversed Kullback-Leibler divergence in Example 4}
\begin{centering}
{\footnotesize{}}%
\begin{tabular}{cccccccccccc}
\hline
 & \multicolumn{5}{c}{{\footnotesize{$\begin{array}{c}\!\!\!\!\!
{}\tilde{H}_{\{i\},TV}^{(L,S,m)},\, L=10^{6} \\
(\mathrm{total\: variational\: distance})
\end{array}$}}} &  & \multicolumn{5}{c}{{\footnotesize{$\begin{array}{c}\!\!\!\!\!
{}\tilde{H}_{\{i\},KL'}^{(L,S,m)},\, L=10^{6} \\
(\mathrm{Kullback-\: Leibler\:\mathrm{divergence}})
\end{array}$}}}\tabularnewline
\cline{2-6} \cline{8-12}
 & \multicolumn{2}{c}{{\footnotesize{$\begin{array}{c}\!\!\!\!\!
\mathrm{Univariate\: PDD}\\
(S=1)^{\mathrm{(a)}}
\end{array}$}}} &  & \multicolumn{2}{c}{{\footnotesize{$\begin{array}{c}\!\!\!\!\!
\mathrm{Bivariate\: PDD}\\
(S=2)^{\mathrm{(b)}}
\end{array}$}}} &  & \multicolumn{2}{c}{{\footnotesize{$\begin{array}{c}\!\!\!\!\!
\mathrm{Univariate\: PDD}\\
(S=1)^{\mathrm{(a)}}
\end{array}$}}} &  & \multicolumn{2}{c}{{\footnotesize{$\begin{array}{c}\!\!\!\!\!
\mathrm{Bivariate\: PDD}\\
(S=2)^{\mathrm{(b)}}
\end{array}$}}}\tabularnewline
\cline{2-3} \cline{5-6} \cline{8-9} \cline{11-12}
{\footnotesize{$\begin{array}{c}\!\!\!\!\!
\mathrm{Random}\\\!\!\!\!\!
\mathrm{variable}
\end{array}$}} & \!\!\!{\footnotesize{$m=2$}} & \!\!\!{\footnotesize{$m=3$}} &  & \!\!\!{\footnotesize{$m=2$}} & \!\!\!{\footnotesize{$m=3$}} &  & \!\!\!{\footnotesize{$m=2$}} & \!\!\!{\footnotesize{$m=3$}} &  & \!\!\!{\footnotesize{$m=2$}} & \!\!\!{\footnotesize{$m=3$}}\tabularnewline
\hline
\!\!\!{\footnotesize{$P_{H}$}} & \!\!\!{\footnotesize{0.021}} & \!\!\!{\footnotesize{0.021}} &  & \!\!\!{\footnotesize{0.020}} & \!\!\!{\footnotesize{0.020}} &  & \!\!\!{\footnotesize{0.005}} & \!\!\!{\footnotesize{0.005}} &  & \!\!\!{\footnotesize{0.005}} & \!\!\!{\footnotesize{0.005}}\tabularnewline
\!\!\!{\footnotesize{$P_{V}$}} & \!\!\!{\footnotesize{0.085}} & \!\!\!{\footnotesize{0.085}} &  & \!\!\!{\footnotesize{0.085}} & \!\!\!{\footnotesize{0.085}} &  & \!\!\!{\footnotesize{0.017}} & \!\!\!{\footnotesize{0.017}} &  & \!\!\!{\footnotesize{0.016}} & \!\!\!{\footnotesize{0.016}}\tabularnewline
\!\!\!{\footnotesize{$E$}} & \!\!\!{\footnotesize{0.085}} & \!\!\!{\footnotesize{0.085}} &  & \!\!\!{\footnotesize{0.088}} & \!\!\!{\footnotesize{0.088}} &  & \!\!\!{\footnotesize{0.016}} & \!\!\!{\footnotesize{0.016}} &  & \!\!\!{\footnotesize{0.019}} & \!\!\!{\footnotesize{0.019}}\tabularnewline
\!\!\!{\footnotesize{$\nu$}} & \!\!\!{\footnotesize{0.018}} & \!\!\!{\footnotesize{0.018}} &  & \!\!\!{\footnotesize{0.018}} & \!\!\!{\footnotesize{0.018}} &  & \!\!\!{\footnotesize{0.004}} & \!\!\!{\footnotesize{0.004}} &  & \!\!\!{\footnotesize{0.004}} & \!\!\!{\footnotesize{0.004}}\tabularnewline
\!\!\!{\footnotesize{$u_{xF}$}} & \!\!\!{\footnotesize{0.436}} & \!\!\!{\footnotesize{0.436}} &  & \!\!\!{\footnotesize{0.434}} & \!\!\!{\footnotesize{0.434}} &  & \!\!\!{\footnotesize{0.345}} & \!\!\!{\footnotesize{0.345}} &  & \!\!\!{\footnotesize{0.344}} & \!\!\!{\footnotesize{0.344}}\tabularnewline
\!\!\!{\footnotesize{$u_{yF}$}} & \!\!\!{\footnotesize{0.084}} & \!\!\!{\footnotesize{0.084}} &  & \!\!\!{\footnotesize{0.083}} & \!\!\!{\footnotesize{0.083}} &  & \!\!\!{\footnotesize{0.007}} & \!\!\!{\footnotesize{0.007}} &  & \!\!\!{\footnotesize{0.007}} & \!\!\!{\footnotesize{0.007}}\tabularnewline
\!\!\!{\footnotesize{$u_{xG}$}} & \!\!\!{\footnotesize{0.441}} & \!\!\!{\footnotesize{0.441}} &  & \!\!\!{\footnotesize{0.438}} & \!\!\!{\footnotesize{0.438}} &  & \!\!\!{\footnotesize{0.344}} & \!\!\!{\footnotesize{0.344}} &  & \!\!\!{\footnotesize{0.343}} & \!\!\!{\footnotesize{0.343}}\tabularnewline
\!\!\!{\footnotesize{$u_{yG}$}} & \!\!\!{\footnotesize{0.085}} & \!\!\!{\footnotesize{0.085}} &  & \!\!\!{\footnotesize{0.085}} & \!\!\!{\footnotesize{0.085}} &  & \!\!\!{\footnotesize{0.007}} & \!\!\!{\footnotesize{0.007}} &  & \!\!\!{\footnotesize{0.007}} &\!\!\! {\footnotesize{0.007}}\tabularnewline
\!\!\!\!\!{\footnotesize{No. of FEA}} & \!\!\!{\footnotesize{25}} & \!\!\!{\footnotesize{33}} &  & \!\!\!{\footnotesize{277}} & \!\!\!{\footnotesize{481}} &  & \!\!\!{\footnotesize{25}} & \!\!\!{\footnotesize{33}} &  & \!\!\!{\footnotesize{277}} & \!\!\!{\footnotesize{481}}\tabularnewline
\hline
\end{tabular}
\par\end{centering}{\footnotesize \par}

{\scriptsize{~~~~~~~~~(a) No. of FEA = $N(m+1)+1$.}}{\scriptsize \par}

{\scriptsize{~~~~~~~~~(b) No. of FEA = $N(N-1)(m+1)^2/2+N(m+1)+1$.}}
\label{table8}
\end{table}

Table \ref{table8} presents the approximate univariate sensitivity indices $\tilde{H}_{\{i\},TV}^{(L,S,m)}$ (total variational distance) and $\tilde{H}_{\{i\},KL'}^{(L,S,m)}$ (reversed Kullback-Leibler divergence) of the maximum von Mises stress by the PDD-KDE-MC method. The PDD expansion coefficients were estimated by $S$-variate dimension-reduction integration \cite{xu04}, requiring one- ($S=1$) or at most two-dimensional ($S=2$) Gauss quadratures. The order $m$ of orthogonal polynomials and number $n$ of Gauss quadrature points in the dimension-reduction numerical integration are $2\le m\le3$ and $n=m+1$, respectively. The indices are broken down according to the choice of selecting $S=1,2$ and $m=2,3$.  In all PDD approximations, the sample size $L=10^6$. The sensitivity indices by the PDD-KDE-MC methods in Table \ref{table8} quickly converge with respect to $S$ and/or $m$. Since FEA is employed for response evaluations, the computational effort of the PDD-KDE-MC method comes primarily from numerically determining the PDD expansion coefficients. The expenses involved in estimating the PDD coefficients vary from 25 to 33 FEA for the univariate PDD approximation and from 277 to 481 FEA for the bivariate PDD approximation, depending on the two values of $m$. Based on the sensitivity indices in Table \ref{table8}, the horizontal boundary conditions ($u_{xF}$ and $u_{xG}$) are highly important; the vertical load ($P_{V}$), elastic modulus ($E$), and vertical boundary conditions ($u_{yF}$ and $u_{yG}$) are slightly important; and the horizontal load ($P_{H}$) and Poisson's ratio ($\nu$) are unimportant in influencing the maximum von Mises stress.

It is important to recognize that the respective univariate and bivariate PDD solutions in this particular problem are practically the same.  Therefore, the univariate PDD solutions are not only accurate, but also highly efficient. This is because of a realistic example chosen, where the individual main effects of input variables on the von Mises stress are dominant over their interactive effects. Finally, this example also demonstrates the non-intrusive nature of the PDD-KDE-MC method, which can be easily integrated with commercial or legacy computer codes for analyzing large-scale complex systems.

\section{Conclusion}
A general multivariate sensitivity index, referred to as the $f$-sensitivity index, is presented for global sensitivity analysis. The index is founded on the $f$-divergence, a well-known divergence measure from information theory, between the unconditional and conditional probability measures of a stochastic response.  The index is applicable to random input following dependent or independent probability distributions.  Since the class of $f$-divergence subsumes a wide variety of divergence or distance measures, numerous sensitivity indices can be defined, affording diverse choices to sensitivity analysis.  Several existing sensitivity indices or measures, including mutual information, squared-loss mutual information, and Borgonovo's importance measure, are shown to be special cases of the proposed sensitivity index. A detailed theoretical analysis reveals the $f$-sensitivity index to be non-negative and endowed with a range of values, where the smallest value is \emph{zero}, but the largest value may be finite or infinite, depending on the generating function $f$ chosen.  The index vanishes or attains the largest value when the unconditional and conditional probability measures coincide or are mutually singular. Unlike the variance-based Sobol index, which is invariant only under affine transformations, the $f$-sensitivity index is invariant under nonlinear but smooth and uniquely invertible transformations.  If the output variable and a subset of input variables are statistically independent, then there is no contribution from that subset of input variables to the sensitivity of the output variable.  For a metric divergence, the resultant $f$-sensitivity index for a group of input variables increases from the unconditional sensitivity index for a subgroup of input variables, but is limited by the residual term emanating from the conditional sensitivity index.

Three new approximate methods, namely, the MC, KDE-MC, and PDD-KDE-MC methods, are proposed to estimate the $f$-sensitivity index.  The MC and KDE-MC methods are both relevant when a stochastic response is inexpensive to evaluate, but the methods depend on how the probability densities of a stochastic response are calculated or estimated. The PDD-KDE-MC method, predicated on an efficient surrogate approximation, is relevant when analyzing high-dimensional complex systems, demanding expensive function evaluations. Therefore, the computational burden of the MC and KDE-MC methods can be significantly alleviated by the PDD-KDE-MC method.  In all three methods developed, the only requirement is the availability of input-output samples, which can be drawn either from a given computational model or from actual raw data.  Numerical examples, including a computationally intensive stochastic boundary-value problem, demonstrate that the proposed methods provide accurate and economical estimates of density-based sensitivity indices.

\end{document}